\documentclass[a4paper,12pt,reqno]{amsart}
\usepackage{amsmath}
\usepackage{amsfonts}
\usepackage{amssymb}
\usepackage{amsthm}
\usepackage{amscd}
\usepackage{color}
\usepackage{verbatim}
\usepackage{eucal,url,amssymb,enumerate,amscd,}
\usepackage{amsfonts}
\usepackage[all]{xy}

\usepackage{amsmath,amsthm,amssymb,amscd,enumerate,eucal,url}
\usepackage{verbatim}

\newcommand{\cT}{\mathcal{T}}
\newcommand{\cR}{\mathcal{R}}
\newcommand{\la}{\langle}
\newcommand{\ra}{\rangle}

\newcommand{\Gl}{\mathrm{Gl}}

\newcommand{\R}{\mathbb{R}}
\newcommand{\ZZ}{\mathbb{Z}}
\newcommand{\ups}{\upsilon}
\newcommand{\e}{\varepsilon}
\newcommand{\G}{\Gamma}
\newcommand{\x}{\times}
\newcommand{\UU}{\mathrm{U}}
\newcommand{\Fix}{\operatorname{Fix}}

\newcommand{\Sp}{\mathrm{Sp}}
\newcommand{\HH}{\mathbb{H}}

\newcommand{\CC}{\mathbb{C}}
\newcommand{\RR}{\mathbb{R}}
\newcommand{\PP}{\mathbb{P}}
\newcommand{\Z}{\mathbb{Z}}
\newcommand{\p}[1]{\mathbb{P}\left(#1\right)}
\newcommand{\CP}{\mathbb{CP}}

\newcommand{\Spin}{\mathrm{Spin}}
\newcommand{\SU}{\mathrm{SU}}
\newcommand{\SO}{\mathrm{SO}}

\newcommand{\Gtwo}{\mathrm{G}_2}

\newcommand{\jota}{\mathfrak{j}}
\newcommand{\inc}{\hookrightarrow}
\newcommand{\rId}{\mathrm{Id}}
\newcommand{\rJ}{\mathrm{J}}

\newcommand{\rr}{\mathrm{r}}

\newcommand{\frh}{\mathfrak{h}}
\newcommand{\frg}{\mathfrak{g}}

\DeclareMathOperator{\coker}{Coker}

\DeclareMathOperator{\Diff}{Diff}

\newcommand{\rb}{\mathrm{b}}
\newcommand{\rqr}{\mathrm{q}}
\newcommand{\bx}{\mathbf{x}}

\newcommand{\ssf}{\mathsf{f}}
\newcommand{\sg}{\mathbf{g}}
\newcommand{\ui}{\underline{\i}}
\newcommand{\uk}{\underline{\k}}
\renewcommand{\a}{\alpha}
\renewcommand{\b}{\beta}
\renewcommand{\d}{\delta}
\newcommand{\g}{\gamma}
\renewcommand{\e}{\varepsilon}

\renewcommand{\l}{\lambda}
\renewcommand{\k}{\kappa}
\newcommand{\m}{\mu}
\newcommand{\n}{\nu}
\renewcommand{\o}{\omega}
\renewcommand{\p}{\phi}

\newcommand{\vp}{\varphi}

\newcommand{\s}{\sigma}
\renewcommand{\t}{\tau}

\renewcommand{\G}{\Gamma}
\renewcommand{\O}{\Omega}
\renewcommand{\S}{\Sigma}

\renewcommand{\L}{\Lambda}
\renewcommand{\P}{\Phi}

\newcommand{\h}{\theta}

\newcommand{\rT}{\mathrm{T}}

\theoremstyle{plain}
\newtheorem{proposition}{Proposition}

\newtheorem{theorem}[proposition]{Theorem}
\newtheorem{lemma}[proposition]{Lemma}
\newtheorem{corollary}[proposition]{Corollary}
\theoremstyle{definition}
\newtheorem{definition}[proposition]{Definition}

\newtheorem{remark}[proposition]{Remark}

\renewcommand{\Re}{\mathfrak{Re}\,}

\title[A compact non-formal closed $\Gtwo$ manifold with $\lowercase{b}_1=1$]{A compact non-formal closed $\Gtwo$ manifold with $b_1=1$}

\author[L. Mart\' in-Merch\' an]{Luc\'ia  Mart\' in-Merch\' an}
\address{Facultad de Ciencias, Universidad
de Málaga, Bulevar Louis Pasteur, 31, 29010, Málaga, Spain}
\email{lmmerchan@uma.es}



\begin{document}

\begin{abstract}
We construct a compact manifold with a closed $\Gtwo$ structure
not admitting any torsion-free $\Gtwo$ structure, which is
non-formal and has first Betti number $b_1=1$.
We develop a method of resolution for orbifolds
that arise as a quotient $M/\ZZ_2$ with $M$
a closed $\Gtwo$ manifold 
under the  assumption that
the singular locus carries a nowhere-vanishing closed $1$-form. 
\end{abstract}

\maketitle

\section{Introduction}\label{sec:intro}

A $\Gtwo$ structure on a $7$-dimensional manifold $M$ is a reduction
of the structure group of its frame bundle to the exceptional Lie group $\Gtwo$.
Such a structure determines an orientation, a metric $g$ and a non-degenerate $3$-form $\vp$; 
these define a cross product $\x$ on $TM$ by means of the expression
$$
\vp(X,Y,Z)=g(X \x Y,Z).
$$

The group $\Gtwo$ appears on Berger's list \cite{BE} of possible holonomy groups of simply connected, irreducible and non-symmetric Riemannian manifolds.  Non-complete metrics with holonomy $\Gtwo$ were given by Bryant in \cite{Br87} and complete metrics were obtained by Bryant and Salamon in \cite{BrS89}. First compact examples were constructed in 1996 by Joyce in \cite{J1} and \cite{J2}. More compact manifolds with holonomy $\Gtwo$ were constructed later by Kovalev \cite{Kovalev}, Kovalev and Lee \cite{KovalevLee}, Corti, Haskins, Nordstr\"om and Pacini \cite{CHNP} and recently by Joyce and Karigiannis \cite{JK}. 

The torsion of a $\Gtwo$ structure $(M,\vp,g)$ is defined as $\nabla \vp$, the covariant
derivative of $\vp$. Fern\' andez and Gray \cite{FG} classified $\Gtwo$ structures into $16$
different types  according to
equations involving the torsion of the structure. In this paper we focus on two of them, namely {\em torsion-free} and {\em closed} $\Gtwo$
structures. A $\Gtwo$ structure is called torsion-free if the holonomy of $g$ is contained in $\Gtwo$, that is $\nabla \vp=0$ or equivalently $d\vp=0$ and $d\star \vp=0$, where $\star$ denotes the Hodge star. A $\Gtwo$ structure is said to be closed if it verifies $d\vp=0$; these are also named {\em calibrated}.
Metrics defined by such types of $\Gtwo$ structures have interesting properties; while torsion-free $\Gtwo$ manifolds are Ricci-flat, closed $\Gtwo$
manifolds have negative scalar curvature and both the scalar-flatness and the Einstein condition
are equivalent to the fact that the structure is torsion-free (see \cite{Br06} and \cite{CI}).

This paper contributes to understanding topological properties of compact manifolds
 with a closed $\Gtwo$ structure that cannot
be endowed with a torsion-free $\Gtwo$
structure.
First examples of these were provided
by Fern\' andez in \cite{F} and \cite{F-2};
the example in \cite{F} is a nilmanifold and the examples in \cite{F-2} are solvamifolds.
Nilmanifols and solvmanifolds arise as
compact quotients of Lie groups
by lattices; these Lie groups are nilpotent in the first case and solvable in the
second. In both examples the $\Gtwo$ structure
is induced by a closed left-invariant $\Gtwo$ form
on the Lie group. The solvmanifolds in \cite{F-2} 
 have $b_1=3$.
In \cite{CF} the authors classify nilpotent Lie algrebras that admit a closed
$\Gtwo$ structure; this list provides more examples of compact manifolds with $b_1 \geq 2$
endowed with a closed $\Gtwo$ structure but not admitting torsion-free $\Gtwo$ structures.
In \cite{VM} the author develops a method that allows to construct $7$-dimensional
solvable Lie groups endowed with a closed
$\Gtwo$ structure and as an application provided an example with $b_1=1$.
Recently in \cite{FFKM} the autors construct another example that
has $b_1 = 1$. Their
starting point is a nilmanifold $M$ with $b_1=3$ that
admits a closed $\Gtwo$ structure and an involution that preserves it. The quotient $X=M/\ZZ_2$ is an
orbifold with $b_1=1$ and its isotropy
locus consists of $16$ disjoint tori. Then they resolve the singularities to obtain a smooth manifold.

Being this the geography of such manifolds, 
this paper provides an example
of 
a compact manifold carrying a closed $\Gtwo$
structure. Its topological properties
are different from those that the already mentioned ones have, as we shall discuss later.
Our construction consists in resolving an orbifold; for that purpose
we first develop a resolution method that is summarized in the following
result:

\begin{theorem} \label{theo:resol-exist}
Let $(M,\vp,g)$ be a closed $\Gtwo$ structure on a compact manifold. Suppose that $\j \colon M \to M$ is an involution such that $\j^*\vp=\vp$ and consider the orbifold $X=M/\j$.
Let $L=\Fix(\j)$ be the singular locus of $X$ and suppose that there is a nowhere-vanishing closed $1$-form $\h \in \O^1(L)$. Then, there exists a compact
$\Gtwo$ manifold endowed with a closed $\Gtwo$ structure $(\widetilde X, \widetilde \vp, \widetilde g)$ and a map $\rho \colon \widetilde X \to X$ such that: 
\begin{enumerate}
\item The map $\rho  \colon \widetilde X- \rho^{-1}(L) \to X-L$ is a diffeomorphism.
\item There exists a small neighbourhood $U$ of $L$ such that $\rho^*(\vp)=\widetilde \vp$ on $\widetilde X- \rho^{-1}(U)$. 
\end{enumerate}
\end{theorem}

The fixed point locus $L$ is an oriented $3$-dimensional manifold (see Lemma \ref{lem:3-dim}); the existence of a nowhere-vanishing closed $\h \in \O^1(L)$ is equivalent to the fact that
each connected component of $L$ is a mapping torus of an orientation-preserving diffeomorphism of an oriented surface. In our example,
the singular locus
is formed by $16$ disjoint nilmanifolds whose universal
covering is the Heisenberg group. 



 The resolution method follows the
ideas of Joyce and Karigiannis in \cite{JK}, where they develop a method to resolve $\ZZ_2$ singularities
induced by the action of an involution on manifolds endowed with a torsion-free $\Gtwo$  structure in the
case that the singular locus $L$ has a nowhere-vanishing harmonic $1$-form.
The local model of the singularity being $\RR^3 \x (\CC^2/\{\pm 1 \})$, the resolution is constructed by replacing a tubular neighbourhood
of the singular locus a with a bundle over $L$ with fibre the
Eguchi-Hanson space. Then they construct a $1$-parameter family of
closed
$\Gtwo$ structures on the resolution; these have small torsion when the value of the parameter
is small. Then they apply a theorem of Joyce \cite[Th. 11.6.1]{Joyce2} which states that
if one can find a closed $\Gtwo$ structure $\vp$
on a compact $7$-manifold $M$ whose torsion is sufficiently small in a certain
sense, then there exists a torsion-free $\Gtwo$ structure which is close to $\vp$
and it determines the same de Rham cohomology class. This method provides a torsion-free $\Gtwo$ structure on the resolution; if its fundamental group is finite then its holonomy
is $\Gtwo$.

The main difficulty of their construction relies on the fact that two of the three 
pieces that they glue, namely an annulus around the singular set of the orbifold 
and a germ of resolution, do not come naturally equipped with a torsion-free
$\Gtwo$ structure. However, there is a canonical way to define a $\Gtwo$ structure on them
and to obtain a closed $\Gtwo$ structure by making a small perturbation. The torsion of 
the structure is
too large so that they need to make additional corrections.
We shall follow the same ideas to perform the resolution; the method is simplified because
we avoid these technical difficulties.

In this paper we are interested in the interplay
between closed $\Gtwo$ manifolds with small first Betti number 
and the condition of being formal.
Formal manifolds are those whose 
rational cohomology algebra 
is described by its rational model.
 This is a notion of rational homotopy theory and
has been sucessfully applied in some geometric situations. 
The Thurston-Weinstein problem
is a remarkable example in the context of symplectic geometry; this consists
in constructing symplectic manifolds with no K\" ahler structure.
Deligne, Griffiths, Morgan and Sullivan proved in \cite{DGMS} that compact K\" ahler manifolds are formal; thus,
non-formal symplectic manifolds are solutions of this problem.
Formality is less understood in the case of exceptional holonomy; 
in particular,
the problem
of deciding whether or not manifolds with holonomy $\Gtwo$ and $\Spin(7)$ 
are formal is still open. There are some partial
results for holonomy $\Gtwo$ manifolds; in \cite{CN} authors proved 
that compact non-formal manifolds with holonomy $\Gtwo$ have
second Betti number $b_2 \geq 4$.
In addition, in \cite{CKT} authors proved that compact
manifolds with holonomy $\Gtwo$ are \textit{almost formal}; this condition implies
that triple Massey products $\la \xi_1,\xi_2,\xi_3 \ra$
are trivial except perhaps for the case
that the degree of $\xi_1$, $\xi_2$ and $\xi_3$ is $2$. 
 Non-trivial Massey products are obstructions to formality but there are examples
of non-formal compact $7$-manifolds that only have trivial triple Massey products (see \cite{CN}).
However, the presence of a geometric structure makes the situation different; for instance in \cite{MT} the authors prove that simply-connected $7$-dimensional Sasakian manifolds are formal if and only if its
triple Massey products are trivial.

Formal examples of closed $\Gtwo$ manifolds that do not admit any torsion-free $\Gtwo$ structure are the solvmanifolds provided in \cite{F-2} and \cite{VM}, and the compact manifold with $b_1=1$ provided in \cite{FFKM}. Non-formal examples are the nilmanifolds obtained in \cite{CF}; these have $b_1\geq 2$. 
In this paper we prove:
\begin{theorem}
There exists a compact non-formal closed $\Gtwo$ manifold with $b_1=1$ that cannot be endowed with a
torsion-free $\Gtwo$ structure.
\end{theorem}
The manifold $\widetilde X$ that we construct is the resolution of a closed $\Gtwo$ orbifold $X$, obtained as the quotient of a nilmanifold $M$ by the action of the group $\ZZ_2$. The orbifold has $b_1=1$ and a non-trivial Massey product coming from $M$. The resolution process does not change the first Betti number; in addition the non-trivial Massey product on $X$ lifts to a non-trivial Massey product on $\widetilde X$.

This paper is organized as follows. In section \ref{sec:pre} we review some necessary preliminaries on orbifolds, $\Gtwo$ structures and formality. Section \ref{sec:resolu} is devoted to prove Theorem \ref{theo:resol-exist}, and in section \ref{sec:topo} we characterise the cohomology ring of the resolution.
With these tools at hand we finally construct in section \ref{sec:const} the non-formal compact closed $\Gtwo$ manifold with $b_1=1$.

\noindent\textbf{Acknowledgements.} I am grateful to my thesis advisors Giovanni Bazzoni and Vicente Mu\~noz for suggesting this this problem to me and for useful conversations.
I acknowledge financial support by a FPU Grant (FPU16/03475).

\section{Preliminaries} \label{sec:pre}

\subsection{Orbifolds}
We first introduce some aspects about orbifolds, which can be found in \cite{CFM} and \cite{MR}.
\begin{definition}\label{def:orbifold}
An $n$-dimensional orbifold is a Hausdorff and second countable space $X$ 
endowed with an atlas $\{(U_{\a},V_\a, \psi_{\a},\G_{\a})\}$, 
where $\{V_\a\}$ is an open cover of $X$,
$U_\a \subset\RR^n$, $\G_{\a} < \Diff(U_\a)$ is a finite group acting by diffeomorphisms, and 
$\psi_{\a}\colon U_{\a} \to V_{\a} \subset X$ is a $\G_{\a}$-invariant map which induces a 
homeomorphism $U_{\a}/\G_\a \cong V_{\a}$. 

There is a condition of compatibility of charts for intersections.
For each point $x \in V_{\a} \cap V_{\b}$ there is some $V_\d \subset V_{\a} \cap V_{\b}$ 
with $x \in V_\d$ so that there are group monomorphisms $\rho_{\d \a}\colon \G_\d \inc \G_\a$,
$\rho_{\d \b}\colon \G_\d \inc \G_\b$, and open differentiable embeddings $\imath_{\d \a}\colon U_\d \to U_\a$, 
$\imath_{\d \b}\colon U_\d \to U_\b$, which satisfy 
$\imath_{\d \a}(\g(x))=\rho_{\d \a}(\g)(\imath_{\d \a}(x))$ and 
$\imath_{\d \b}(\g(x)) = \rho_{\d \b}(\g)(\imath_{\d \b}(x))$, for all $\g\in \G_\d$. 
\end{definition}

We can refine the atlas of an orbifold $X$ in order to obtain better properties; given a point $x \in X$, there is a chart $(U,V,\psi,\G)$ with $U \subset \RR^n$, $U/\G \cong V$,
so that the preimage $\psi^{-1}(\{x\})= \{u\}$,
and verifies $\g(u)=u$ for all $\g \in \G$. We call $\G$ the \emph{isotropy group} at $x$, 
and we denote it by $\G_x$. This group is well defined up to conjugation by a diffeomorphism
of a small open set of $\RR^n$.
The singular locus of $X$ is the set
$
S=\{x\in X \mbox{ s.t. } \G_x\neq \{1\} \},
$
and of course, $X-S$ is a smooth manifold.

We now describe the de Rham complex of an $n$-dimensional orbifold $X$. First of all, a $k$-form $\eta$ on $X$ consists of a collection
of differential $k$-forms $\{\eta_\a \}$ such that:
\begin{enumerate}
\item $\eta_\a \in \O^k(U_\a)$ is $\G_\a$-invariant,
\item If $V_\d \subset V_\a $ and $\imath_{\d \a} \colon U_\d \to U_\a$ is the associated embedding, then $\imath_{\d \a}^*(\eta_\a)= \eta_\d$.
\end{enumerate}
The space of orbifold $k$-forms on $X$ is denoted by $\O^k(X)$. In addition, it is obvious that the wedge product of
orbifold forms and the exterior differential $d$ on $X$ are well defined. Therefore  $(\O^*(X),d)$  is a
 differential graded algebra that we call the de Rham complex of $X$. Its cohomology coincides
with the cohomology of the space $X$ with real coefficients, $H^*(X)$ (see \cite[Proposition 2.13]{CFM}).

In this paper the orbifold involved is the orbit space of a smooth manifold $M$ under the action of $\Z_2= \{ \rId, \j \}$, where $\j$ is an involution. The singular locus of $X=M/\Z_2$ is $\Fix(\j)$. In addition, let us denote by $\O^k(M)^{\Z_2}$ the space of $\Z_2$-invariant $k$-forms. Then
$$
\O^k(X)=\O^k(M)^{\Z_2},
$$
and both the wedge product and exterior derivative preserve the $\Z_2$-invariance.
An averaging argument ensures that $H^k(X)=H^k(M)^{\Z_2}$.

\subsection{$\Gtwo$ structures} We now focus on $\Gtwo$ structures on manifolds and orbifolds. Basic references are \cite{Br06}, \cite{FG}, \cite{HL}, \cite{Joyce2} and \cite{Salamon}.

Let us identify $\RR^7$ with the imaginary part of the octonions $\mathbb{O}$. The multiplicative structure on $\mathbb{O}$ endows $\RR^7$ with a cross product $\x$, which defines a $3$-form  $\vp_0(u,v,w)= \la u\x v, w \ra $, where  $\la \cdot, \cdot \ra$ denotes the scalar product on $\RR^7$. In coordinates,
\begin{equation}\label{eqn:std}
\vp_0= v^{127} + v^{347} + v^{567} + v^{135} - v^{236} - v^{146} - v^{245},
\end{equation}
where $(v^1,\dots,v^7)$ is the standard basis of $(\RR^7)^*$ and $v^{ijk}$ stands for $v^i \wedge v^j \wedge v^k$.
The stabilizer of $\vp_0$ under the action of $\Gl(7,\RR)$ on $\L^3(\RR^7)^*$ is the group $\Gtwo$, a simply connected $14$-dimensional Lie group which is contained in $\SO(7)$. 

\begin{definition}
Let $V$ be a real vector space of dimension $7$. A $3$-form $\vp \in \L^3 V^*$ is
a $\Gtwo$ form on $V$ if there is a linear isomorphism $u\colon V\to \RR^7$
such that $u^*(\vp_0)= \vp$, where $\vp_0$ is given by equation (\ref{eqn:std}).
\end{definition}

A $\Gtwo$ structure $\varphi$ determines an orientation because $\Gtwo \subset \SO(7)$; the choice of a volume form $\mathrm{vol}$ on $V$ compatible with the orientation determines a unique metric $g_{\mathrm{vol}}$ with associated unit-length volume form $\mathrm{vol}$ by the formula:
$$
i(x)\vp \wedge i(y) \vp \wedge \vp = 6 g_{\mathrm{vol}}(x,y)\mathrm{vol},
$$
which ensures that the metric $u^*(g_0)$ is determined by the volume form $u^*(\mathrm{vol}_{\RR^7})$.
Note that the metric $u^*(g_0)$ does not depend on the isomorphism $u$ with $u^*(\vp_0)=\vp$. We say that $g=u^*(g_0)$ is the metric associated to $\vp$.
Of course, a $\Gtwo$ form $\vp$ induces a cross product $\x$ on $V$ by the formula $\vp(u,v,w)=g(u\x v,w)$.

The orbit of $\vp_0$ under the action of $\Gl(7,\RR)$ is an open set of $\L^3 (\RR^7)^*$, thus the space of $\Gtwo$ forms on $\RR^7$ is an open set.

\begin{definition}
Let $M$ be a $7$-dimensional manifold. A $\Gtwo$ form on $M$ is a $3$-form $\vp \in \O^3(M)$ such that for every $p \in M$
the $3$-form $\vp_p$ is a $\Gtwo$ form. 

Let $X$ be a $7$-dimensional orbifold with atlas $\{(U_\a,V_\a,\psi_\a,\G_\a) \}$. A $\Gtwo$ form on $X$ is a
differential $3$-form $\vp \in \O^3(X)$ such that $\vp_\a$ is a $\Gtwo$ form
on $U_\a$.
\end{definition}

Let $\vp$ be a $\Gtwo$ form on a manifold $M$ or an orbifold $X$. In both cases, 
$\vp$ determines a metric $g$ and a cross product $\x$. In this case we say that $(M,\vp,g)$ or $(X,\vp,g)$
is a $\Gtwo$ structure. In addition, $\Gtwo$ manifolds are of course oriented.

We state a well-known fact about $\Gtwo$ structures (see for instance \cite[Chapter 10, Section 3]{Joyce2}).

\begin{lemma}\label{universal}
There exists a universal constant $m$ such that if $(M,\varphi,g)$ is a $\Gtwo$ structure and $\| \phi - \varphi \|_{C^0,g}< m$ then $\phi$ is a $\Gtwo$ form.
\end{lemma}
\begin{proof}
Let $(\R^7, \varphi_0,g_0)$ be the standard $\Gtwo$ structure. Being the space of $\Gtwo$ forms on $\R^7$  open in $\Lambda^3 (\R^7)^*$, there exists a constant $m>0$ such that if a $3$-form $\phi_0$ verifies that $\| \phi_0 - \varphi_0 \|_{g_0} <m$, then $\phi_0$ is a $\Gtwo$ form. We now check that $m$ is the claimed universal constant.
Let $(M,\varphi,g)$ be a $\Gtwo$ manifold; let $\phi$ such that $ \| \phi_p - \varphi_p \|_{g_p} < m$ for every $p\in M$. In order to check that $\phi_p$ is a $\Gtwo$ form, let $A \colon (T_pM, \varphi_p,g_p) \to (\R^7, \varphi_0,g_0)$ be an isomorphism of $\Gtwo$ vector spaces, then:
$$ 
\| A^t \phi_p - \varphi_0 \|_{g_0} =  \| \phi_p - \varphi_p \|_{g_p} < m
$$
and therefore $A^t \phi_p$ is a $\Gtwo$ form. Since $A$ is an isomorphism, $\phi_p$ is also a $\Gtwo$ form.
\end{proof}

In \cite{FG} Fern\' andez and Gray classified $\Gtwo$ structures $(M,\vp, g)$ into $16$ types
according to $\nabla \vp$, where $\nabla$ denotes the Levi-Civita connection associated to $g$.
The motivation for such classification is the holonomy principle, stating that the holonomy of $g$
is contained in $\Gtwo$ if and only if $\nabla \vp=0$. In \cite{FG} they also prove that
$\nabla \vp=0$ if and only if $d\vp=0$ and $d(\star \vp)=0$, where $\star$ denotes the Hodge star.

In this paper we are interested in closed and torsion-free $\Gtwo$ structures on manifolds and orbifolds:

\begin{definition}
Let $(M,\vp,g)$ or $(X,\vp,g)$ a $\Gtwo$ structure on a manifold or an orbifold.
We say the $\Gtwo$ structure is closed if $d\vp=0$. If in addition $d(\star \vp)=0$ we say that the
$\Gtwo$ structure is torsion-free.
\end{definition}

\begin{definition}
Let $(X,\vp)$ be a closed $\Gtwo$ structure on a $7$-dimensional orbifold. A
closed $\Gtwo$ resolution of $(X, \vp)$ consists of a smooth manifold endowed with a closed $\Gtwo$ structure $(\widetilde X, \p)$ 
and a map $\rho \colon \widetilde X \to X$ such that:
\begin{enumerate}
\item Let $S\subset X$ be the singular locus and $E=\rho^{-1}(S)$. Then, $\rho|_{\widetilde X- E} \colon \widetilde X - E \to X- S$ is a diffeomorphism,
\item Outside a neighbourhood of $E$, $\rho^*(\vp)= \p$.
\end{enumerate}
The subset $E$ is called exceptional locus.
\end{definition}

\subsubsection{$\Gtwo$ involutions}

\begin{definition}
Let $(M,\vp)$ be a $\Gtwo$ manifold, we say that $\j \colon M \to M$ is a $\Gtwo$ involution if $\j^*(\vp)=\vp$, $\j^2=\rId$, and $\j \neq \rId$.
\end{definition}

In this paper we shall focus on orbifolds that are obtained as a quotient of a closed $\Gtwo$ manifold $(M,\vp)$ by the action of a $\Gtwo$ involution $\j$; that is $X=M/\j$. The next result states that the fixed locus $L$ of $\j$ is a $3$-dimensional submanifold. 

\begin{lemma} \label{lem:3-dim}
The submanifold $L$ is $3$-dimensional and oriented by $\vp|_L$. In addition, $\vp|_L$ is the oriented unit-length volume form determined by the metric $g|_L$.
\end{lemma}
\begin{proof}
The result is deduced from the fact that if $(\RR^7,\vp_0, \la \cdot, \cdot \ra)$ is the standard $\Gtwo$ structure on $\RR^7$ and if $\jota \in \Gtwo$ is an involution,  $\jota \neq \rId$, then $\jota$ is diagonalizable with eigenvalues $\pm 1$ and $\dim(V_{1})=3$, $\dim(V_{-1})=4$, where $V_{\pm 1}$ denotes the eigenspace associated to the eigenvalue $\pm 1$. In addition, $\vp_0(v_1,v_2,v_3)=\pm 1$ if $(v_1,v_2,v_3)$ is an orthogonal basis of $V_1$.

We now prove this statement; first $\jota$ is diagonalizable with eigenvalues $\pm 1$ because $\jota^2=\rId$, $\jota \neq \rId$ and $\j \in \SO(7)$. Let us take a unit-length vector $v_1 \in V_{1}$; the vector space $W=\la v_1 \ra^\perp$ is fixed by $\jota$ because $\jota\in \SO(7)$, and carries in addition an $\SU(3)$ structure
determined by $\o= i(v_1)\vp_0$, $\Re(\O)=\vp_0|_W$ (see \cite{SW}). Of course, the $\SU(3)$ structure is preserved by $\jota$.
Viewed as a complex map,  $\jota \colon W \to W$ has three complex eigenvalues $\l_1, \l_2, \l_3$ that verify $\l_j^2=1$ and $\l_1 \l_2 \l_3=1$  because $\jota^2=\rId$ and $\jota$ preserves the $\SU(3)$ structure. Being $\jota \neq \rId$, we obtain that $\l_1=1$ and $\l_2=\l_3=-1$ up to a permutation of the indices; this proves that $\dim(V_{1})=3$ and $\dim(V_{-1})=4$.  Now observe that $\jota (u \x v)= \jota(u) \x \jota(v)$, where $\x$ is the cross product on $\RR^7$ that determines $\vp$. Thus,
let $(v_1,v_2,v_3)$ be an orthogonal basis of $V_1$, then $v_1 \x v_2 \in V_1$; so necessarily, $v_1 \x v_2= \pm v_3$ and $\vp_0(v_1,v_2,v_3)=\pm 1$.
\end{proof}

\begin{remark}
If $d\vp=0$, Lemma \ref{lem:3-dim} states that $L$ is a calibrated submanifold of $M$ in the sense of \cite{HL}.
\end{remark}

\subsubsection{$\SU(2)$ structures}

Let us identify $\RR^4$ with $\HH$ and identify $\SU(2)$ with $\Sp(1)$ as usual.
The multiplication by $i$, $j$ and $k$ on the quaternions yields $\Sp(1)$-equivariant endomorphisms $I$, $J$ and $K$ that determine
invariant $2$-forms by the contraction of these endomorphism with
the scalar product on $\RR^4$.
In coordinates, these are:
\begin{equation}\label{eqn:su2}
\begin{aligned}
\o_1^0 = w^{12} + w^{34}, \qquad
\o_2^0= w^{13} - w^{24}, \qquad
\o_3^0 = w^{14} + w^{23}.
\end{aligned}
\end{equation}
where $(w_1,w_2,w_3,w_4)$ denotes the standard basis of $\RR^4$.

\begin{definition}
Let $W$ be a real vector space of dimension $4$. An $\SU(2)$ structure on $W$
is determined by $2$-forms $(\o_1,\o_2,\o_3)$ such that 
there is a linear isomorphism $u\colon W \to \RR^4$
with $u^*(\o_j^0)= \o_j$, where  the forms $\o_j^0$ are given by equation (\ref{eqn:su2}).
\end{definition}

An $\SU(2)$ structure on a vector space $W$ determines a $\Gtwo$ structure on $W\oplus \RR^3 $.
To check this we can suppose that $(W,\o_1,\o_2,\o_3)=(\RR^4,\o_1^0,\o_2^0, \o_3^0)$. Denote by $(v_5,v_6,v_7)$ the standard basis of $\RR^3$, then
\begin{equation}\label{eqn:su2-g2}
\vp_0= v^{567} + \o_1^0 \wedge v^7 + \o_2^0 \wedge v^5 - \o_3^0 \wedge v^6.
\end{equation}
In addition if we fix on $\RR^3$ the orientation determined by $v^{567}$, then $W$ is oriented by $\frac{1}{2}(\o_1^0)^2$.

\begin{definition}
Let $N$ be a $4$-dimensional manifold. An $\SU(2)$ structure on $N$ consists of
$2$-forms $(\o_1,\o_2,\o_3) \in \O^2(N)$ that determine an $\SU(2)$ structure on $T_pN$ for every $p\in N $. In addition, if $d\o_1=d\o_2=d\o_3=0$ we say that $(\o_1,\o_2,\o_3)$ is a hyperK\" ahler structure.

Let $Y$ be a $4$-dimensional orbifold with atlas $\{(U_\a,V_\a,\psi_\a,\G_\a) \}$. An $\SU(2)$ structure on $Y$ consists of
$2$-forms $(\o_1,\o_2,\o_3) \in \O^2(Y)$ such that $(\o_1^\a,\o_2^\a,\o_3^\a)$ is an $\SU(2)$ structure on $U_\a$. In addition, if $d\o_1=d\o_2=d\o_3=0$ we say that $(\o_1,\o_2,\o_3)$ is a hyperK\" ahler structure.
\end{definition}

In view of Lemma \ref{lem:3-dim} the local model of $X$ around $L$ is $( \CC^2/\ZZ_2) \x \RR^3$, with $\Z_2= \la -\rId, \rId \ra$. The standard $\Gtwo$ form induces the orbifold hyperK\" ahler $\SU(2)$ structure $(\o_1^0, \o_2^0, \o_3^0)$ on $\CC^2/\Z_2$. We now detail the hyperK\" ahler resolution of $Y=\CC^2/\Z_2$; this will be useful in order to construct the resolution of $X$ in section \ref{sec:resolu}.

The holomorphic resolution of $Y$ is $N=\widetilde \CC^2/\ZZ_2$; where $\widetilde \CC^2$ is the blow-up of $\CC^2$ at $0$. That is,
$$
\widetilde \CC^2= \{ (z_1,z_2,\ell)\in \CC^2 \x \CP^1 \mbox{ s.t. } (z_1,z_2)\in \ell\},
$$
and the action of $-\rId$ lifts to $(z_1,z_2,\ell) \longmapsto (-z_1,-z_2,\ell)$.
We shall call the exceptional divisor $E=\{0 \}\x \CP^1 \subset N$.
Note that there is a well-defined projection $\s_0 \colon N \to \CP^1$. Let us consider $r_0 \colon Y \to [0,\infty)$ the radial function induced from $\CC^2$; one can check taking coordinates that $r_0^2$ is not smooth on $N$, but $r_0^4$ is.

Consider the blow up map, $\chi_0 \colon N \to Y$. Then, one can check that $\chi_0^*(\o_2^0)$ and $\chi^*(\o_3^0)$ are non-degenerate smooth forms on $N$; this holds because $\o_2^0 + i \o_3^0 = dz_1 \wedge dz_2$ and the pullback of a holomorphic form under a holomorphic resolution is holomorphic.

A computation in coordinates shows that $\chi_0^*(\o_1^0)$ has a pole on $\CP^1$. Let $a>0$ and define $\ssf_a(x)=  \sg_a(x) + 2a\log(x)$, where $ \sg_a(x)=(x^4 + a^2)^{1/2}- a\log((x^4 + a^2)^{1/2}+a)$. Consider on $Y-E$:
$$
\widehat \o_1^a= -\frac{1}{4} dId\ssf_a(r_0).
$$
One can check that $(\widehat \o_1^a,\chi_0^*(\o_2^0), \chi_0^*(\o_3^0))$ is a hyperK\" ahler structure on $N-E$; it can be extended as a hyperK\" ahler structure on $N$ because:
$$ 
-\frac{1}{4} dId (\log(r_0^2))= \s_0^*(\o_{\CP^1}),
$$
where $\o_{\CP^1}$ stands for the Fubini-Study form of $\CP^1$.

\subsection{Formality}

In this section we review some definitions and results about formal manifolds and formal
orbifolds; basic references are \cite{DGMS}, \cite{FOT}, and \cite{OT}.

We work with commutative differential graded algebras (in the sequel CDGAs); these consist
of a pairs $(A,d)$ where $A$ is a commutative graded algebra $A=\oplus_{i\geq 0}{A^i}$ over $\RR$, and $d \colon A^* \to A^{*+1}$ is a differential, which is a graded derivation that verifies $d^2=0$.
If $a \in A$ is an homogenous element, we denote its degree by $|a|$, and $\bar{a}=(-1)^{|a|}a$.

The cohomology algebra of a CDGA $(A,d)$ is denoted by $H^*(A,d)$; it is also a CDGA with the differential being zero. If $a\in A$ is a closed element we denote its cohomology class by $[a]$. The CDGA $(A,d)$ is said to be connected if $H^0(A,d)=\RR$.

 In our context, the main examples of CDGAs are the de Rham complex of a manifold or an orbifold. In section \ref{sec:const} we also make use of the Chevalley-Eilenberg CDGA of a Lie group $G$, that consists of the algebra $\L^* \frg^*$, the differential of a $1$-form is
 $
 d\a (x,y) = -\a[x,y],
 $ 
and is extended to $\L^* \frg^*$ as a graded derivation.

\begin{definition}
A CDGA $(A,d)$ is said to be minimal if:
\begin{enumerate}
\item $A$ is free as an algebra, that is $A$ is the free algebra
$\L V$  over a graded vector
space $V =\oplus_i V^i$.
\item There is a collection of generators $\{ a_i \}_{i }$ indexed by a well ordered set, such
that $|a_i| \leq  |a_j|$ if $i<j$ and each $da_j$ is expressed in terms of the previous $a_i$ with $i<j$.
\end{enumerate}
\end{definition}

Morphisms between CDGAs are required to preserve the degree and to commute with the differential; a morphism of CDGAs $\k \colon (B,d) \to (A,d)$ is said to be a quasi-isomorphism if it induces an isomorphism on cohomology $\k \colon H^*(B,d) \to H^*(A,d)$.

\begin{definition}
A CDGA $(B,d)$ is a model of the CDGA $(A,d)$ if there exists a quasi-isomorphism
$
\k \colon (B, d) \to (A,d).
$
If $(B,d)$ is minimal we say that $(B,d)$ is a minimal model of $(A,d)$.
\end{definition}

Minimal models  of connected DGAs exist and are unique up to isomorphism of CDGAs. So we
define the minimal model of a connected manifold or a connected orbifold as the minimal model of
its associated de Rham complex.

\begin{definition}
A minimal algebra $(\L V , d)$ is formal if there exists a quasi-isomorphism,
$$
(\L V, d ) \to (H^*( \L V, d), 0).
$$
A manifold or an orbifold is formal if its minimal model is formal.
\end{definition}

We now recall the definition of triple Massey products; these are objects that detect non-formality of manifolds. 
Let $(A,d)$ be a CDGA and let $\xi_1,\xi_2,\xi_3$ be cohomology classes such that $\xi_1\xi_2=0$ and $\xi_2\xi_3=0$. Under these assumptions we can define the triple Massey product of these cohomology classes $\la \xi_1,\xi_2,\xi_3 \ra$. In order to provide its definition we first introduce the concept of a defining system for $ \la \xi_1,\xi_2, \xi_3 \ra $.
\begin{definition}
 A defining system for $ \la \xi_1,\xi_2, \xi_3 \ra $ is an element $(a_1,a_2,a_3,a_{12},a_{23})$
such that:
\begin{enumerate}
\item $[a_i]=\xi_i$ for $1 \leq i \leq 3$,
\item $da_{12}= \bar{a}_1a_2$, and $da_{23}= \bar{a}_2a_3$.
\end{enumerate}
\end{definition}
One can check that
$
\bar{a}_1a_{23} + \bar{a}_{12} a_3
$
is a closed ($|a_1| + |a_2| + |a_3| -1$)-form. The triple Massey product $\la \xi_1, \xi_2,\xi_3 \ra$ is the set formed by the cohomology classes that defining systems determine, that is:
$$
 \{[ \bar{a}_1a_{23} + \bar{a}_{12} a_3]  \mbox{ s.t. } (a_1,a_2,a_3,a_{12},a_{23}) \mbox{ runs over all defining systems}\}.
$$
If $0 \in \la \xi_1, \xi_2, \xi_3 \ra$ we say that the triple Massey product is trivial.

\begin{theorem}
Let $(\L V,d)$ be a formal minimal algebra. Let $\xi_1,\xi_2,\xi_3$ be cohomology classes such that the triple Massey product
$\la \xi_1,\xi_2,\xi_3\ra$ is defined.
Then $\la \xi_1,\xi_2,\xi_3 \ra$ is trivial.
\end{theorem}

As a consequence, we obtain: 
\begin{corollary}\label{cor:massey-form}
Let $(\L V,d)$ be the minimal model of $(A,d)$. Let $\xi_1,\xi_2,\xi_3 \in H^*(A,d)$ such that the triple Massey product
$\la \xi_1,\xi_2,\xi_3\ra$ is defined.
If $ \la \xi_1, \xi_2, \xi_3 \ra$ is not trivial then $(\L V,d)$ is not formal.
\end{corollary}
\begin{proof}
Suppose that $(\L V, d)$ is formal and let 
$
\k \colon (\L V,d)\to (A, d)
$
be a quasi-isomorphism. Let us take cohomology
classes $\xi_1',\xi_2',\xi_3' \in H^*(\L V,d)$ with $\k(\xi_j')=\xi_j$ then the Massey product $\la \xi_1',\xi_2', \xi_3' \ra$ is well-defined and there is a defining system $(a_1,a_2,a_3,a_{12},a_{23})$
such that
$$
\bar{a}_1 a_{23} + \bar{a}_{12} a_3= d\a.
$$ 
But of course $0=\k[\bar{a}_1 a_{23} + \bar{a}_{12} a_3] \in \la \xi_1, \xi_2, \xi_3 \ra$;  yielding a contradiction.
\end{proof}

We finally outline some aspects about finite group actions on minimal models. 
Let $M$ be a compact manifold and 
let $\k \colon (\L V, d) \to (\O(M),d)$ be the minimal model.
Let $\G$ be a finite subgroup of $\Diff(M)$ acting
on the left; the pullback of forms defines a right action of $\G$ on $(\O(M),d)$.

Lifting theorems for CDGAs ensure the existence 
of a morphism $\overline{\g} \colon \L V \to \L V$ that lifts up to homotopy the pullback by 
each $\g \in \G$; that is, $\k \circ \overline{\g} \sim \g^* \circ \k$; in particular, $[\k (\overline{\g}(a))]=[\g^* \k (a)]$ if $da=0$.
This implies that $\overline{\rId} \sim \rId$ and that 
$\overline{\g \g'}\sim \overline{\g} \, \overline{\g}'$; therefore these
liftings provide an homotopy action on $\L V$. These liftings
can be modified making use of group cohomology techniques (see \cite[Theorem 2]{O}) in order to
endow $\L V$ with a right action of $\G$.

\begin{theorem} \label{thm:action-dga}
Let $M$ be a compact connected manifold and let $\G$ be a subgroup of $\Diff(M)$ acting
on the left.

There is a right action of $\G$ on the minimal model $\k \colon (\L V, d) \to (\O(M),d)$ by morphisms of CDGAs such that $[\k(a \g)]=[\g^* \k(a)]$ for every closed element $a \in \L V$ and every $\g \in \G$.
\end{theorem}

If there is a right action of a finite group $\G$ on a CDGA $(A,d)$ one can consider the CDGA of $\G$-invariant elements $(A^\G,d)$. An average argument leads us to $H^*(A,d)^\G=H^*(A^\G,d)$. 
In addition, if $\G$ also acts on $(B,d)$ on the right by morphisms and 
$\i \colon (A,d) \to (B,d)$ is a morphism such that $[\i (a \g)]= [(\i a) \g]$ for every closed $a \in A$ and
$\g \in \G$ one can define:
$$ 
\ui \colon (A^\G ,d)\to (B^\G,d), \qquad \ui a= |\G |^{-1}\sum_{\g  \in \G}{\i(a) \g},
$$
where $|\G|$ denotes the cardinal number of $\G$. This verifies that $[\ui (a)]=[\i(a)]$ for closed elements $a \in A^\G$.
In particular if $\i$ is a quasi-isomorphism so is $\ui$.

\begin{lemma}\label{lem:formal-quotient}
Let $\G$ be a finite group acting on a compact connected manifold $M$ by diffeomorphisms.
If $M$ is formal then $M/\G$ is also formal.
\end{lemma}
\begin{proof}
First of all, the fact that
$(\O (M/\G),d)= (\O(M)^\G,d)$ and our previous argument ensure that $H^*(M/\G)=H^*(M)^\G$. Let $\k \colon (\L V, d) \to (\O(M),d)$ be the minimal model of $M$ as constructed in Theorem \ref{thm:action-dga}. The CDGA $((\L V)^\G, d)$ is a model for $(\O(M/\G),d)$ because of the quasi-isomorphism
 $\uk \colon ((\L V)^\G, d) \to (\O(M)^\G,d)$ defined as above. Consider $(\L W,d)$ the minimal model of $(\O(M/\G),d)$ and let  $\psi \colon (\L W, d) \to ((\L V)^\G, d) $ be a quasi isomorphism.

Being $M$ formal one can consider a quasi-isomorphism $\i \colon (\L V, d) \to (H^*(\L V,d),0)$ and define $\ui \colon ((\L V)^\G, d) \to (H^*(\L V,d)^\G,0)=(H(\L W, d),0)$, which is also a quasi-isomorphism. Then we can construct a quasi isomorphism:
$$
\ui \circ \psi  \colon (\L W, d) \to (H^*(\L W, d),0).
$$
Therefore, $M/\G$ is formal.
\end{proof}

\section{Resolution process} \label{sec:resolu}

Let $(M,\vp,g)$ be a closed $\Gtwo$ structure on a compact manifold $M$, let $\j \colon M \to M$ be a $\Gtwo$ involution and let $X=M/\j$. The singular locus of the closed $\Gtwo$ orbifold $(X, \varphi, g)$ is the set $L=\Fix(\j)$, a $3$-dimensional oriented manifold according to Lemma \ref{lem:3-dim}. This section is devoted to construct a resolution $\rho \colon \widetilde{X} \to X$ under the extra assumption that $L$ has a nowhere-vanishing closed $1$-form $\h \in \O^1(L)$.

This hypothesis yields a topological characterisation of $L$ that we now outline. Let us denote by $L_1,\dots, L_r$ the connected components of $L$; according to Tischler's Theorem \cite{T} each $L_i$ is a fibre bundle over $S^1$ with fibre a connected surface $\S_i$; that is, $L_i$ is the mapping torus of a diffeomorphism $\psi_i \in \Diff(\S_i)$:
$$
L_i = \S_i \x [0,1]/ (x,0) \sim (\psi_i(x),1).
$$
Let us denote $\rqr_i \colon \S_i \x [0,1] \to L_i$ the quotient map and $\rb_i \colon L_i \to S^1$ the bundle map; then we can suppose that $\h|_{L_i}= \rb_i^*(\h_0)$, where $\h_0$ denotes the angular form on $S^1$. 
In addition, taking into account that $L_i$ is oriented and that $H^3(L_i)\cong \{ [\a]\in H^2(\S_i) \mbox{ s.t. } \psi_i^*[\a]=[\a] \}$,
 we obtain that $\S_i$ is oriented and $\psi_i^*= \rId$ on $H^2(\S_i)$. 

The resolution process consists of replacing a neighbourhood of $L$ with a closed $\Gtwo$ manifold. The local model of the singularity is $\RR^3 \x Y$ where $Y=\CC^2/\ZZ_2$ as we discussed in section \ref{sec:pre}. The closed $\Gtwo$ manifold that we introduce is the blow-up of $\nu/\j$ at the zero section, where $\n$ denotes the normal bundle of $L$ in $M$. Its
 local model is $\RR^3 \x N$ where $N=\widetilde{\CC}^2/\ZZ_2$. This requires the choice of complex structure on $\nu/\j$ which is determined by a choice of a unit-lenght vector $V$ on $L$ by means of the expression $I(X)=V \x X$, where $\x$ is the cross-product associated to $\vp$. This vector field exists because $L$ is parallelizable, but we shall choose $V= \| \h \|^{-1} \h^\sharp$ in order to guarantee that the $\Gtwo$ form that we later define on the resolution is closed.

Before constructing a $\Gtwo$ form on the resolution we study the $O(1)$ term of $\exp^*(\vp)$ by splitting $T\n$ into an horizontal and a vertical bundle with the aid of a connection. This allows us to obtain a formula for the $O(1)$ term that resembles the standard $\Gtwo$ structure on $\RR^3 \x Y$.
Its pullback under the blow-up map has a pole at the zero section; a non-singular $\Gtwo$ structure is defined on the
resolution following the ideas we introducd in subection \ref{subsec:local} for resolving the local model. This form is not closed in general, so that we need to consider a closed approximation of it. In addition, the resolution process requires the introduction of a $1$-parameter family of closed forms; small values of the parameter guarantee that these are non-degenerate and closed to $\exp^*(\vp)$ on an annulus around $L$ after a diffeomorphism. As Remark \ref{rem:size-exc} states, the size of the exceptional divisor decreases as the parameter tends to $0$.

This section is organized as follows: in subsection \ref{subsec:split} we introduce some notations concerning the normal bundle $\nu$ of $L$ and  we understand its second order Taylor approximation $\p_2$ in subsection \ref{subsec:taylor}; this is an auxiliary construction.
In subsection \ref{subsec:local} we obtain local formulas for the $O(1)$-terms and introduce the parameter $t$; these tools allow us to perform the resolution in \ref{subsec:resol}.

\subsection{Splitting of the normal bundle} \label{subsec:split}

We now introduce some notations that we need for the resolution process.  Let $\pi \colon \nu \to L$ be the normal bundle of $L$. We consider $R>0$ such that the neighbourhood of the $0$ section $Z$, $\nu_R=\{ v_p \in \nu_p \mbox { s.t. } \|v_p\|<R \}$ is diffeomorphic to a neighbourhood $U$ of $L$ on $M$ via the exponential map. On $\nu_R$ we consider $\p=(\exp)^*\varphi$, which is a closed $\Gtwo$ form on $\nu_R$. In addition, the induced involution on $\nu$ is $d\j(v_p)=-v_p$; but we shall also denote it by $\j$.
It shall be useful to denote the dilations by $F_t \colon \nu \to \nu$, $F_t(v_p)=tv_p$. We also define the vector field over $\nu$, $\cR(v_p)=\frac{d}{dt}\Big|_{t=0} e^tv_p$.

A connection $\nabla$ on $\nu$ induces a splitting $T\nu= V\oplus H$ where $V=\ker(d\pi)\cong \pi^*\nu$ and $d\pi_{v_p} \colon H_{v_p} \to T_pL$ is an isomorphism; being $TM|_L= \nu \oplus TL$, the connection induces an isomorphism $\cT \colon T\nu \to \pi^*(TM|_L)$. The choice of $\nabla$ is made in subsection \ref{subsec:resol}.

Note that any tensor $\rT$ on $TM|_{L}$ defines a tensor on $\pi^*(TM|_L)$  because $\pi^*(TM|_L)_{v_p}= T_pM|_{L}$
. Using this we define on $\nu$:

\begin{enumerate}
\item A metric, $g_1= \cT^*(g|_{L})$; that is, $g_1$ makes $(H_{v_p},g_1)$ and $(T_pL,g)$ isometric, $H_{v_p}$ is perpendicular to $V_{v_p}$ and $V_{v_p}$ isometric to $\nu_p$. 
\item A $\Gtwo$ structure $\p_1= \cT^*(\varphi|_L)$ with $g_1$ as an associated metric.
\end{enumerate}
Of course, $\cT$ is an isometry.
These tensors are constant in the fibres in the following sense; under the identification $\widehat \cT_{v_p}=\cT^{-1}_{0_p} \circ  \cT_{v_p} \colon T_{v_p}\nu \to T_{0_p}\nu$ it holds that $\widehat  \cT_{v_p}^*(g_1)=g_1$ and $\widehat \cT_{v_p}^*(\p_1)=\p_1$. Note also that these values coincide with $\exp^*g|_{Z}$ and $\phi$ respectively since $(d\exp)|_{Z}=\mathrm{Id}$. These tensors are thus independent of $\nabla$ only on $Z$. 

We shall also denote $W^k_{i,j}= \L^i V^* \otimes \L^j H^*$ where we understand $V^*= \mathrm{Ann}(H)$ and $H^*= \mathrm{Ann}(V)$. There are $g_1$-orthogonal splittings
 $\L^k T^*\nu = \oplus_{i+j=k} W^{k}_{i,j}$ and given $\a \in \L^k T^*\n$ we denote by $[\a]_{i,j}$ the projection of $\a$ to $W_{i,j}$. 

Observe also that one can restrict each $\b \in \L^k V^*$ to the fibre $\n_p$, and the restriction $\rr_k\colon \L^k V^* \to \L^k T^*\nu$, $\rr_k(\beta)_{v_p}= \beta_{v_p}|_{\n_p}$ is an isomorphism because $T_{v_p}\nu_p=V_{v_p}$.

We now state some technical observations concerning vertical forms; proofs are computations in terms of local coordinates that we include for completeness.

\begin{remark}\label{remark-connection} Note that $H^* = \pi^* (T^*L)$ does not depend on the connection but $V^*$ does. More precisely, in local coordinates $(x_1,x_2,x_3,y_1,y_2,y_3,y_4)\in U \times \R^4$ the horizontal distribution at $(x,y)$ is generated by:
$$
\partial_{ x_i}- \sum_{j=1}^{4} A_i^j(x,y) \partial_{y_j},
$$
where $A_i^j(x,y)=\sum_{k=1}^4 {A_{i,k}^j(x)y_k}$ for some differentiable functions $A_{i,k}^j$. Then $V^*$ is generated by:
$$
\eta_j= dy_j + \sum_{i=1}^3{A_i^j(x,y) dx_i}.
$$
Note also that since $A_i^j(x,ty)=tA_i^j(x,y)$ we get that $F_t^*(\eta_i)=t\eta_i$.
\end{remark}

\begin{lemma} \label{lem:homog-tensors} The following identities hold:
\begin{enumerate}
\item $F_t^*(\p_1)= [\p_1]_{0,3} + t^2[\p_1]_{2,1}$
\item $F_t^*(g_1)= g_1|_{H\otimes H} + t^2 g_1|_{V\otimes V}$
\end{enumerate}
\end{lemma}
\begin{proof}
We shall prove the first equality being the second similar. Note that $\p_1|_Z$ is a $\Gtwo$ structure whose induced metric makes $V$ perpendicular to $H$ and $H|_{Z}=TZ$; thus taking into account formula (\ref{eqn:su2-g2}) we can write in local coordinates:
$$
\p_1|_Z = f(p)dx_1\wedge dx_2 \wedge dx_3 + \sum_{i=1}^3{\sum_{j<k}{f_{ijk}(p)dx_i\wedge dy_j \wedge dy_k }}.
$$
Thus, $\p_1 = [\p_1]_{0,3} + [\p_1]_{2,1}$, where $([\p_1]_{0,3})_{v_p}= f(p)dx_1\wedge dx_2 \wedge dx_3$ and $([\p_1]_{2,1})_{v_p}= \sum_{i=1}^3{\sum_{j<k}{f_{ijk}(p)dx_i\wedge (\eta_j)_{v_p}\wedge (\eta_k)_{v_p} }}$. Therefore, $F_t^*([\p]_{0,3})= [\p]_{0,3}$ and according to the previous remark, $F_t^*[\p]_{2,1}=t^2[\p]_{2,1}$.
\end{proof}

\begin{lemma} \label{zero-section}
Let $\mu \in V^*$ be a form such that $\mu=0$ on $T\n|_Z$. Then,  $[d\m]_{1,1}=0$ and $[d\m]_{0,2}=0$ on $T\n|_Z$.
\end{lemma}
\begin{proof}
In local coordinates, $\m= \sum_{i=1}^{4}{f_i(x,y)\eta_i}$ with $f_i(x,0)=0$ as  $\mu=0$ on $T\nu|_Z$. Then, 
\begin{align*}
d\m=& \sum_{i=1}^4 \sum_{j=1}^3 {\frac{\partial f_i}{\partial x_j }(x,y) dx_j \wedge \eta_i} \\
 &+ \sum_{i=1}^4 \sum_{j=1}^4 {\frac{\partial f_i}{\partial y_j }(x,y) dy_j \wedge \eta_i} + \sum_{i=1}^4 f_i(x,y)d\eta_i .
\end{align*}
Since $f_i(x,0)=0$ and $\eta_i|_{T\n|_Z}= dy_i$ the following equalities hold on $T\n|_Z$:
\begin{align*}
[d\m]_{2,0} (x,0)=& \sum_{i=1}^4 \sum_{j=1}^4 {\frac{\partial f_i}{\partial y_j }(x,0) dy_j \wedge dy_i}, \\
[d\m]_{1,1} (x,0)=&  \sum_{i=1}^4 \sum_{j=1}^3{\frac{\partial f_i}{\partial x_j} (x,0) dx_j \wedge \eta_i}= 0, \\ 
[d\m]_{0,2} (x,0)=&0.
\end{align*}
\end{proof}

\begin{lemma} \label{lem:bounded norm} Consider coordinates $(x,y)=(x_1,x_2,x_3,y_1,y_2,y_3,y_4)\in B \times \R^4$ of $\n$, with $B\subset \RR^3$ a closed ball. Let $\eta_j$ be the projection of $dy_j$ to $V^*$ as in Remark \ref{remark-connection}. Then,
$
\|(\eta_i)_{(x,0)}\|_{g_1}= \|(\eta_i)_{(x,y)}\|_{g_1}
$
and $
\|(dx_i)_{(x,0)}\|_{g_1}= \|(dx_i)_{(x,y)}\|_{g_1}
$.

There exist $C_1>0$, $C_2>0$ such that
$
\|[d\eta_i]_{0,2}\|_{g_1} \leq C_1 r 
$
and $ \|[d\eta_i]_{1,1}\|_{g_1} \leq C_2$ on $\n$.
\end{lemma}
\begin{proof}
The first two equalities are clear taking into account that $\cT^*(\eta_j)=\eta_j$, $\cT^*(dx_j)=dx_j$ and that $\cT$ is a $g_1$-isometry. For the third and fourth equality we first compute $d\eta_j$
$$
d\eta_j= \sum_{k=1}^4{ \sum_{i,l=1}^3{y_k \frac{\partial  A_{i,k}^j(x)}{\partial x_l}dx_l \wedge dx_i}}
+ \sum_{k=1}^4 \sum_{i=1}^3  A_{i,k}^j(x) dy_k \wedge dx_i.
$$
This implies that:
\begin{align*}
[d\eta_j]_{0,2}=& \sum_{k=1}^4{ \sum_{i,l=1}^3{y_k \frac{\partial A_{i,k}^j(x)}{\partial x_l}dx_l \wedge dx_i}} - \sum_{k,n=1}^4 \sum_{i,m=1}^3  A_{i,k}^j(x) A_{m,n}^k(x) y_n dx_m \wedge dx_i, \\
[d\eta_j]_{1,1}= & \sum_{k=1}^4 \sum_{i=1}^3  A_{i,k}^j(x) \eta_k \wedge dx_i.
\end{align*}
The functions $|A_{i,k}^j|$ are bounded on $B$, and that the $g_1$-norm of the terms $\eta_m \wedge dx_j$ and $dx_j \wedge dx_k$ are constant on the fibres as explained before. Taking into account that $L$ is compact the choice of
constants $C_1$ and $C_2$ becoms clear.
\end{proof}

\subsection{Taylor series} 
\label{subsec:taylor}  We now introduce the Taylor series of $\p$ and interpolate it with the seccond order approximation. This is an auxiliary tool for our resolution process.

Consider the dilation over the fibres $F_t \colon\nu \to \nu$, and define the Taylor series of $F_t^*\p$ and $F_t^*g$ near $t=0$ (note that $F_0^*(\p)$ and $F_0^*(g)$ are defined on $\nu$). That is,
$$
 F_t^*(\p) \sim \sum_{k=0}^{\infty}{t^{2k}\p^{2k}}, \quad F_t^*g \sim \sum_{k=0}^\infty {t^{2k}g^{2k}}.
$$
Note that we only wrote even terms because both $\p$ and $g$ are $\j$ invariant and $\j=F_{-1}$. In addition, since $F_{ts}=F_t\circ F_s$ we have that $F_s^*(\p^{2k})=s^{2k}\p^{2k}$, $F_s^*(g^{2k})=s^{2k}g^{2k}$. For $i+j=3$ and $p+q=2$ we define $\p^{2k}_{i,j}= [\p^{2k}]_{i,j}$, $g^{2k}_{p,q}|_{V^p \otimes H^q}$ ; here $V^p$ denotes the tensor product of $V$ with itself $p$ times.

We have the following properties:
\begin{enumerate}
\item $\|\p^{2k}_{i,j} \|_{g_1} = O(r^{2k-i})$, where $r$ is measured with respect to the metric on $\nu$. To check it let $\|v_p\|_{g_1}=1$; taking into account Lemma \ref{lem:homog-tensors} and the fact that $F_t \colon (\n,g_1|_{H\otimes H} +t^2g_1|_{V\otimes V}) \to (\n,g_1)$ is an isometry we get:
\begin{align*}
\| (\p^{2k}_{i,j})_{r v_p} \|_{g_1}=& \| r^{2k} F_{r^{-1}}^* (\p^{2k}_{i,j})_{rv_p} \|_{g_1} 
=
r^{2k}\|(\p^{2k}_{i,j})_{v_p}\|_{g_1|_{H\otimes H} +r^{2}g_1|_{V\otimes V}} \\
=& \, r^{2k-i}\|(\p^{2k}_{i,j})_{v_p}\|_{g_1}.
\end{align*}

\item The previous statement ensures that $\p^{2k}_{i,j}=0$ if $i>2k$.

\item If $k\geq 1$, $\p^{2k}$ is exact.

Being $\p^{2k}$ homogeneous of order $2k$, we have that $\mathcal{L}_{\cR}(\p^{2k})= 2k \p^{2k}$; where $\cR(v_p)=\frac{d}{dt}\Big|_{t=0} (e^t v_p)$ is defined as above. In addition, since $\p$ is closed we have that $d\p^{2k}=0$ for every $k$. Thus,
$2k\p^{2k}=d(i(\cR)\p^{2k})$.
\end{enumerate}

Taking these properties into account we construct a $\Gtwo$ form $\p_{3,\e}$ that interpolates $\p$  with the approximation $\p_2= \p^0 + \p^2$. The parameter $\e>0$ indicates that the interpolation occurs on $r\leq \e$ and is done in such a way that $\p_{3,\e}|_{r \leq \frac{\e}{2}}= \p_2$.
Of course, this is possible because the difference between $\p$ and $\p_2$ is small near the zero section. 

\begin{proposition} \label{prop:interpolation-order-2}
The form  $\p_2=\p^0 + \p^2$ is closed and $\p = \p_2 + O(r)$.
There exists $\e_0>0$ such that for each $\e < \e_0$ there exists a $\j$-invariant $\Gtwo$ form $\p_{3,\e}$ such that $\p_{3,\e}=\p_2$ if $r\leq \frac{\e}{2}$ and $\p_{3,\e}=\p$ if $r\geq \e$.
\end{proposition}
\begin{proof}
The first part is a consequence of the previous remark; zero order terms are $\p^0=\p^0_{0,3}$ and $\p^{2}_{2,1}$, thus $\p=\p_2 + O(r)$. In addition, $\p_2$ is closed because each $\p^{2k}$ is. 

Since $\p|_Z=\p_2|_Z$ Poincar\'e Lemma for submanifolds ensures that $\p=\p_2 + d\xi$ for some $\j$-invariant 2-form $\xi$; more precisely,
$
\xi_{v_x}= \int_{0}^1{i(\mathcal{R}_{\t v_x})(\p - \p_2)d\t}.
$
In addition, $\|\xi\|_{g_1}=O(r^2)$ because $\|\p - \p_2\|_{g_1}=O(r)$. Let $\varpi$ be a smooth function such that $\varpi=1$ if $x\leq \frac{1}{2}$ and $\varpi=0$ if $x\geq 1$ and define $\varpi_\e(x)=\varpi(\frac{x}{\e})$. Then, $|\varpi_\e'| \leq \frac{C}{\e}$ so that 
$$ 
\p_{3,\e}= \p +  d(\varpi_\e(r) \xi) 
$$
is a $\Gtwo$ form on $r\leq \e$ if $\e$ is small enough because it is $O(\e)$-near $\p$. The form $\p_{3,\e}$ interpolates $\p_2$ with $\p$ over the stated domains and it is $\j$-invariant because both $\p$ and $\varpi_\e(r) \xi$ are.
\end{proof}


\subsection{Local formulas} \label{subsec:local}
The purpose of this section is making an additional preparation; we first provide a local formula for $\p_1$ that will be useful in order to construct the $\Gtwo$ form of the resolution. Later
we change $\p_2$ by $O(r)$ terms so that we control its local formula and we introduce the parameter $t$; these preparations are essential to construct a closed $\Gtwo$ form on the resolution.

\subsubsection{Formula for $\p_1$} 

We first write $\p_1$ and $g_1$ in terms of the components of the Taylor series of $g$ and $\p$. This  is an easy consequence of the homogeneus behaviour of the tensors involved:
\begin{lemma} \label{lem:homog}
The following equalities hold:
\begin{enumerate}
\item $\p_1 = \p^0 + \p^2_{2,1} $
\item $ g_1 = g_{0,2}^0 + g_{2,0}^2 $
\end{enumerate}
\end{lemma}
\begin{proof}
We prove the first equality, being the second similar.
Using the fact that $\p^0=\p^0_{0,3}$ and $\p^2_{2,1}$ are homogeneous one can check that these are constant on the fibres. We shall do it for $\p^2_{2,1}$, write in local coordinates $(x,y)$:
$$
\p^2_{2,1}=\sum_{i=1}^3{\sum_{j<k}f_{ijk}(x,y)dx_i\wedge (\eta_j)_{(x,y)}\wedge (\eta_k)_{(x,y)} }.
$$
Taking into account that $F_t^*\p^2_{2,1}=t^2\p^2_{2,1}$ and $F_t^*\eta_i=t\eta_i$ we get $f_{ijk}(x,ty)=f_{ijk}(x,y)$. Therefore, $f_{ijk}(x,y)=f_{ijk}(x,0)$. Since $\p_1|_{TM|_Z}= \p|_{TM|_Z}= (\p^0 + \p^2_{2,1})|_{TM|_Z}$, we obtain that $[\p_1]_{0,3}|_{T\n|_Z}=\p^0|_{T\n|_Z}$ and $[\p_1]_{2,1}|_{T\n|_Z}=\p^2_{2,1}|_{T\n|_Z}$. But these forms are constant on the fibres of the bundle $T\nu \to \nu$, so that the previous equalities hold on $T\n$.
\end{proof}

We now obtain a local formula for $\p_1$. For that purpose let us define  $e_1=\|\h\|^{-1} \h$ and consider an orthonormal oriented frame $(e_1,e_2,e_3)$ of $TL$ on a neighbourhood $U\subset L$. Define also the  $\SU(2)$  structure $(\o_1^L,\o_2^L,\o_3^L)$ on $\nu$ by means of the equality:
$$
\varphi|_{L}= e_1\wedge e_2 \wedge e_3 + e_1\wedge \o_1^L + e_2 \wedge \o_2^L - e_3 \wedge \o_3^L.
$$
More precisely, the complex structure is determined by $\omega_1^L=i(e_1^\sharp)\varphi|_{\n}$, that is $I(X)=e_1^\sharp \times X$ where $\times$ denotes the vector product associated to $\varphi|_L$. The complex volume form is $\o_2^L + i\o_3^L$; note that a counterclockwise rotation of angle $\s$ in the plane $(e_2,e_3)$ changes $\o_2^L + i\o_3^L$ by the complex phase $e^{i\s}$. Using $\cT$ we obtain:
$$
\p_1= \pi^*e_1\wedge \pi^*e_2 \wedge \pi^*e_3 + \pi^*e_1\wedge \o_1 + \pi^* e_2 \wedge \o_2 - \pi^* e_3 \wedge \o_3,
$$
where the forms $\o_j \in \L^2 V^*$ are $\j$-invariant and verify $\o_j|_{Z}= \exp^*(\o_j^L)$. Fixed $p \in L$, $(\o_1 |_{\n_p}, \o_2 |_{\n_p}, \o_3 |_{\n_p})$ determines an $\SU(2)$ structure on the $4$-manifold $\nu_p$ because the restriction $\rr_2$ is an isomorphism. The associated metric on $T\nu_p$ is $g_1|_{\nu_p}$ and the complex form is induced by $I$ on $\nu$ under the canonical isomorphism.

Therefore,
$\o_1 |_{\n_p} = -\frac{1}{4}d_{\n_p}( I[dr^2]_{\n_p})$. In addition, since the complex volume form is $dz_1 \wedge dz_2 = \frac{1}{2}d(z_1dz_2 -z_2dz_1)$ there is a $\j$-invariant $1$-form $\mu \in V^*$ such that $d_{\nu_p}(\mu|_{\n_p})= (\o_2 + i \o_3)|_{\n_p}$ and $\m|_{T\n|_{Z}}=0$. We decompose it as $\m = \m_1 + i \m_2$. 

Being the restriction to the fibre $\rr_2$ a monomorphism, we obtain
$$
\o_1= -\frac{1}{4}[d[Idr^2]_{1,0}]_{2,0}, \qquad \o_2 + i\o_3 = [d\m]_{2,0},
$$
here we also denoted by $I$ the complex structure on $V^*$ determined by the complex structure $I(X)=e_1^\sharp \times X$ on $V=\pi^*(\nu)$, this depends on the splitting. Observe that the complex structure $I$ on $\n$ verifies $\j \circ I = I \circ \j$ and thus, the complex structure on $V^*$ verifies $\j I\a= I \j\a$. In particular, $I\a$ is $\j$-invariant if $\a$ is.

\subsubsection{Changing $\p_2$ by $O(r)$ terms.} 

First of all define the $1$-parameter family 
$$ 
\p_2^t = \p^0 + t^2 \p^2= F_t^*(\p_2).
$$
These forms are well-defined on $\nu$ because $\p^0$ and $\p^2$ are homogeneous.
We now change this $1$-parameter family by $O(r)$ terms so that we have an explicit local formula for it.
Consider the exact $\j$-invariant form: 
$$
\b = -\frac{1}{4} \pi^*\h \wedge d( (\|\h\|^{-1}\circ \pi) I[dr^2]_{1,0}) + d(\pi^*e_2\wedge \mu_2 - \pi^*e_3\wedge \mu_3 ) \in W_{2,1}\oplus W_{1,2}\oplus W_{0,3},
$$ 
and note that $\p_1= \pi^*(e_1\wedge e_2 \wedge e_3) + [\b]_{2,1}$. In addition, $\b$ does not depend on the orthonormal oriented basis $(e_2,e_3)$ of $\la \h^*\ra^\perp$.

We now introduce a $1$-parameter family of closed $\j$-invariant forms:
$$
\widehat{\p}_2^t = \pi^*(e_1\wedge e_2 \wedge e_3) + t^2[\b].
$$ 
We claim that fixed $s>0$ there exists $t_s>0$ such that $\widehat \p_2^t$ is a $\Gtwo$ form on $\n_{2s}$ if $t<t_s$ .  To check this we compare $\widehat \p_{2}^t$ with $F_t^*\p_1$ and use Lemma \ref{universal} to conclude. Denote $g_t=F_t^*(g_1)$ and observe that Lemma \ref{lem:homog} implies that $F_t^*\p_1= \p^0 + t^2 \p^2_{2,1}$ and $g_t= t^2 g_{2,0} + g_{0,2}$, then:
$$
\| F_t^*\p_1 - \widehat \p_2^t \|_{g_t} = t \| [\b]_{1,2}\|_{g_1} + t^2\| [\b]_{0,3} \|_{g_1},
$$
so one can bound  $\| [\b]_{1,2}\|_{g_1}$, $\| [\b]_{0,3} \|_{g_1}$ on $\n_{2s}$ and chose $t_s>0$ such that for each $t<t_s$, $t \| [\b]_{1,2}\|_{g_1} + t^2\| [\b]_{0,3} \|_{g_1} <m$ where $m$ is the universal constant obtained in Lemma \ref{universal}. 

We construct a $\Gtwo$ form $\p_{3,s}^t$ that interpolates $F_t^*\p$ with  $\p_2^t$. The parameter $s>0$ indicates that the interpolation occurs on the disk $r\leq s$ and we require that $\p_{3,\e}|_{r \leq \frac{s}{2}}= \p_2$. In subsection \ref{subsec:resol} we employ large values of the parameter.

\begin{proposition}
There is $\xi \in W_{0,2}$ such that $\|\xi\|_{g_1}=O(r^2)$ and $\p^2 = \b +d\xi$.

Fixed $s>0$ there exists $t_s'>0$ such that for each $t<t_s'$, there is a closed $\j$-invariant $\Gtwo$  form $\widehat \p_{3,s}^t$ on $\n_{2s}$ that coincides with $\widehat \p_2^t$ on $r\leq \frac{s}{2}$ and $\p_2^t$ on $r\geq s$. 
\end{proposition}

\begin{proof}
Write the second term of the Taylor series of $\p$ as $\p^2= \p^2_{2,1}+ \p^2_{1,2} + \p^2_{0,3}$ and note that $\p^2_{2,1}=[\b]_{2,1}$. Being  $\b$ and $\p^2$ closed, we obtain $ d (\p^2_{1,2} + \p^2_{0,3}) =  d([\b]_{1,2} + [\b]_{0,3})$. Poincar\'e Lemma ensures that $\p^2_{1,2} + \p^2_{0,3} =[\b]_{1,2} + [\b]_{0,3} + d\xi $ with 
$$
 \xi_{v_x} = \int_{0}^1{i(\cR_{\t v_x})(\p^2_{1,2} + \p^2_{0,3} -[\b]_{1,2} - [\b]_{0,3})d\t} = \int_{0}^1 {i(\cR_{\t v_x})(\p^2_{1,2}  -[\b]_{1,2})d\t}.
$$
 Hence $\xi \in W_{0,2}$. One can check that $\xi$ is $\j$-invariant by taking into account that $\p^2_{1,2}  -[\b]_{1,2}$ is $\j$-invariant and that $\cR_{t\j(v_x)}=\j(\cR_{tv_x})$. 
 
  In addition $\|\xi\|_{g_1}=O(r^2)$ because $\p^2_{1,2} + \p^2_{0,3}|_{Z}=0$ (these terms are $O(r)$ and $O(r^2)$ in the $g_1$-norm) and $([\b]_{1,2} + [\b]_{0,3})|_Z=0$ according to Lemma \ref{zero-section}. 

Let $\varpi$ be a smooth function such that $\varpi=1$ if $x\leq \frac{1}{2}$ and $\varpi=0$ if $x\geq 1$, and let $\varpi_s(x)=\varpi(\frac{x}{s})$. The form
$\widehat \p_{3,s}^t = \p^0 + t^2 \b +  t^2 d(\varpi_s(r) \xi)$
is  closed and $\j$-invariant; it coincides with $\widehat \p_2^t$ on $r\leq \frac{s}{2}$ and with $\p_2^t$ on $r\geq s$. 

It is clear that $\widehat{\p}_{3,s}^t$ is a $\Gtwo$ form on the region $r\geq s$ for $t<t_s$; we now check that it is also a $\Gtwo$ form on $r\leq s$ for some choice of $t$. We are going to compare $\widehat{\p}{}^t_{3,\e}$ with $F_t^*\p_1$ and use Lemma \ref{universal} to conclude the result.

Since $\varpi_s \xi \in W_{0,2}$ we have that $d(\varpi_s \xi) \in  W_{1,2}\oplus W_{0,3}$. As a consequence $\|t^2 d(\varpi_s(r) \xi) \|_{g_t} \leq t \| d(\varpi_s(r) \xi) \|_{g_1}= t(O(r^2s^{-1}) + O(r))$ so that:
\begin{align*}
\| \widehat \p_{3,s}^t - F_t^*\p_1\|_{g_t} =& \,
t (\| [\b]_{1,2}\|_{g_1} + t\| [\b]_{0,3} \|_{g_1} + O(r^2s^{-1}) + O(r)) \\
\leq & \,  t (\| [\b]_{1,2}\|_{g_1} + \| [\b]_{0,3} \|_{g_1} + O(r)).
\end{align*}
For the last equality we used that $t<1$ and that $r\leq 2s$.
Then $\widehat \p_{3,s}^t$ is a $\Gtwo$ form if the parameter $t<t_s$ verifies 
$$
t \, ( \mathrm{max}_{r\leq 2s}(\| [\b]_{1,2}\|_{g_1} + \| [\b]_{0,3} \|_{g_1} + O(r)) < m
$$
where $m$ is the constant provided by Lemma \ref{universal}.
\end{proof}

\subsection{Resolution of $\nu/\j$}\label{subsec:resol}

The resolution process is inspired in the hyperK\"ahler resolution $N= \widetilde{\CC^2}/\Z_2$ of $Y= \CC^2/\Z_2$ described in subsection \ref{subsec:resol}. Consider the blow-up map $\chi_0 \colon N \to Y$ and the hyperK\" ahler structure $(\widehat \o_1^a, \chi_0^*(\o_2^0), \chi_0^*(\o_3^0))$ on $N$. Recall that $\widehat \o_1^a$ denotes the extension of $ -\frac{1}{4} dId\ssf_a(r_0)$, where $r_0$ is the radial function on $\CC^2$ and:
$$
\ssf_a(x)=  \sg_a(x) + 2a\log(x), \qquad  \sg_a(x)=(x^4 + a^2)^{1/2}- a \log((x^4 + a^2)^{1/2} + a).
$$ 

We now focus in the resolution of $\n/\j$. For that purpose, consider the complex structure $I$ on $\n$ determined by the $2$-form $i(e_1^\sharp)\vp|_\n$ and define $P$ as the fiberwise blow-up of $\nu/\j$ at $0$. That is $P= P_{\UU(2)}(\nu)\times_{\UU(2)} N$, where $P_{\UU(2)}(\nu)$ denotes the principal $\UU(2)$-bundle associated to $\n$. This construction yields projections $\chi \colon P \to \nu/\j$  and $\mathrm{pr}= \bar{\pi} \circ \chi$; where $\bar{\pi}$ denotes the map that $\pi \colon \nu \to L$ induces.

 We also define $Q=\chi^{-1}(0)$; this is a $\CP^1$ bundle over $L$ that can be expressed as $Q=P_{\UU(2)}(\n)\times_{\UU(2)} \CP^1$. Note that there is a projection $\s_0 \colon N \to \CP^1$ that induces a complex line bundle $\s \colon P \to Q$.

A $\j$-invariant tensor on $\nu$ descends to $\nu/\j$ and its pullback by $\chi$ is smooth over $P-Q$, but it may not be smooth on $P$. If the tensor preserves the complex structure $I$ on $P$ then the pullback is smooth on $P$  because $P= P_{U(2)}(\nu)\times_{\UU(2)}  N$.
We choose $\nabla$ such that $\nabla I=0$, so that we can lift $\nabla$ to $P$ and define $TP=V'\oplus H'$; this is compatible with the splitting $T\nu=V\oplus H$. In addition, $\mu_2,\m_3,\omega_1,\omega_2,\omega_3$ induce forms on $\nu/ \j$ and $\chi^*(\mu_k)$, $\chi^*(\omega_k)$ are smooth for $k=\{2,3\}$. We shall also consider $\L^k T^*P= \oplus_{i+j=k} \L^i V' \otimes \L^j H'$ and $[\a]=\sum_{i,j}[\a]_{i,j}$.

In order to define a $\Gtwo$ structure on $P$ we need to find a resolution of $\o_1$. For that purpose denote by $r$ the pullback of the radial function on $\nu$ and define:
$$
\widehat \o_1 =  -\frac{1}{4}d (|\h|^{-1}  I[d\ssf_{|\h|}(r)]_{1,0}),
$$
where $|\h|= \| \h \| \circ \mathrm{pr}$.
 Observe that $\sg_{|\h|}(r)$ is smooth on $P$ because $r^4$ is. In addition, $-\frac{1}{2} dI[d(\log(r^2))]_{1,0}= \s^*(F_Q)$ on $P-Q$, where $F_Q$ is the curvature of the line bundle $\s \colon P \to Q$. Fiberwise it coincides with the Fubini-Study form on $\CP^1$. Note also that $\mathrm{pr}^* \h \wedge [\widehat{\o}_1]_{2,0} = -\frac{1}{4} e_1 \wedge [d(I[d\ssf_{|\h|}]_{1,0})]_{2,0}$.
 
We now define a $\Gtwo$ form $\P_1^t$ which is near $\chi^*(F_t^*\p_1)$ on $r>1$, this is:
$$
\P_{1}^t = \mathrm{pr}^*(e_1\wedge e_2 \wedge e_3) +  t^2[\widehat \b]_{2,1},
$$
where 
$$
  \widehat \b = \mathrm{pr}^* \h \wedge \widehat\o_1 + d(\mathrm{pr}^*e_2 \wedge \chi^*(\m_2) - \mathrm{pr}^*e_3 \wedge \chi^*(\m_3)).
$$
Observe that $\b$ does not depend on the orthonormal oriented basis $(e_2,e_3)$ of $\la \h^* \ra^\perp$.  In addition, the metric induced by $\P_1^1$ on $TP$ has the form $h_1= h_{2,0} + h_{0,2}$ where $h_{2,0}$ and $h_{0,2}$ are metrics on $V'$ and $H'$ respectively. In addition, the metric that $\P_1^t$ induces is $h_t= t^2 h_{2,0} + h_{0,2}$.
We define a family of closed forms:
$$
\P_2^{t}=\mathrm{pr}^*( e_1\wedge e_2 \wedge e_3) + t^2 \widehat \b.
$$
Note that $\P_2^t$ is a $\Gtwo$ structure on $\nu_{2s}$ for some $t<t_s''$. This is ensured by Lemma \ref{universal} because:
$$
\| \P_2^t - \P_1^t \|_{h_t} = t \|[\widehat \b]_{1,2}\|_{h_1} + t^2\|[\widehat \b]_{0,3}\|_{h_1},
$$
and one can bound $\|[\widehat \b]_{1,2}\|_{h_1}$ and $\|[\widehat \b]_{0,3}\|_{h_1}$ on $\chi^{-1}(\n_{2s})$.

The parameter $t$ is devoted to compensate errors introduced by $\|[\widehat \b]_{1,2}\|_{h_1}$ and  $\|[\widehat \b]_{0,3}\|_{h_1}$ that mainly come from the terms $[F_Q]_{1,1}$ and $[F_Q]_{0,2}$, which are zero if and only if the curvature is vertical. 
Lemma \ref{lem:Q-trivial} states that the bundle $Q$ is trivial, $Q=L\x \CP^1$. It might not happen that $P$ is the pullback of $N$ via the projection map $L \x \CP^1 \to \CP^1$; in the case that it is, then $F_Q \in \L^2 V'$.

\begin{proposition} \label{prop:inter-resol}
There exist $s_0>1$, such that for each $s>s_0$ one can find $t_s'''$ such that for each $t<t_s'''$ there is  a closed $\Gtwo$ structure $\P_{3,s}^t$ such that  $\P_{3,s}^t= \P_2^t$ on $r \leq \frac{s}{8}$ and $\P_{3,s}^{t}= \chi^*(\widehat \p_{3,s}^t)$ on $r \geq \frac{s}{4}$. 
\end{proposition}
\begin{proof}
On the anulus $\frac{s}{8} < r < \frac{s}{4}$ we have that:
$$
\P_2^t - \chi^*(\widehat \p_{3,s}^t)= \frac{1}{4} t^2 d(\mathrm{pr}^*e_1 \wedge (I[d(\ssf_{|\h|}(r) - r^2 )]_{1,0})). 
$$
We now let $\varpi$ be a smooth function such that $\varpi=1$ if $x\leq \frac{1}{8}$ and $\varpi=0$ if $x\geq \frac{1}{4}$ and $\varpi_s(x)=\varpi(\frac{x}{s})$; then $|\varpi_\e'| \leq \frac{C}{s}$. Define $\bar{\ssf}_{a}(x)=\ssf_a(x)-x^2 =  \frac{a^2}{(x^4 + a^2)^{1/2} + x^2} - a \log((x^4 + a^2)^{1/2} + a) +  2 a\log(x)$ and 
$$
\xi_s = \varpi_s \mathrm{pr}^*e_1 \wedge (Id[\bar{\ssf}_{\h}(r)]_{1,0}).
$$
 The form $d\xi_s$ lies in $W_{2,1}\oplus W_{1,2}\oplus W_{0,3}$. In order to analyze the $h_1$ norm of each component first observe that $|\bar{\ssf}_a'| =O(x^{-1})$ and $|\bar{\ssf}_a''| =O(x^{-2})$ on $x>1$. 

In addition, note that if $(x,y)=(x_1,x_2,x_3,y_1,y_2,y_3,y_4)\in B \times \RR^4$ is a complex unitary parametrisation; that is, in coordinates $I(x,y)=(x_1,x_2,x_3,-y_2,y_1,-y_4,y_3)$ and the vectors of the frame that the parametrisation determines have lenght one. Moreover, connection forms verify $I\eta_1=-\eta_2$, $I\eta_3=-\eta_4$; to check this one  has to observe that the matrices $(A_{i,k}^j)_{k,j}$ defined in Remark \ref{remark-connection} are complex linear because $\nabla I =0$. Taking this into account, a straightforward computation of the pullback yields the claim. 

Taking these observations and Lemma \ref{lem:bounded norm} into account we obtain that on $r>1$:
 \begin{align*}
 \|[d\xi_s]_{2,1}\|_{h_1} =& \| \varpi_s \mathrm{pr}^*e_1 \wedge [dId[\bar{\ssf}_{|\h|}(r)]_{1,0}]_{2,0}  + [d\varpi_s]_{1,0}\wedge \mathrm{pr}^*e_1 \wedge Id[\bar{\ssf}_{|\h|}(r)]_{1,0}\|_{h_1} \\ =&  O(r^{-2}) + O(r^{-1}s^{-1}) , \\
 \|[d\xi _s]_{1,2}\|_{h_1}=& \|\varpi_s \mathrm{pr}^*e_1 \wedge [dId[\bar{\ssf}_{|\h|}(r)]_{1,0}]_{1,1} + \varpi_s \mathrm{pr}^*(de_1) \wedge Id[\bar{\ssf}_{|\h|}(r)]_{1,0} \\  &+ [d\varpi_s]_{0,1}\wedge \mathrm{pr}^*e_1 \wedge Id[\bar{\ssf}_{|\h|}(r)]_{1,0}\|_{h_1} = O(r^{-1})+  O(s^{-1}), \\
 \|[d\xi_s]_{0,3}\|_{h_1}=&  \|\varpi_s \mathrm{pr}^*e_1 \wedge [dId[\bar{\ssf}_{|\h|}(r)]_{1,0}]_{0,2}\|_{h_1}= O(1).
\end{align*}
We now prove that $\| [d\varpi_s]_{1,0}\wedge \mathrm{pr}^*e_1 \wedge Id[\bar{\ssf}_{|\h|}(r)]_{1,0} \|_{h_1}=O(r^{-1}s^{-1})$ and $\| [d\varpi_\s]_{0,1}\wedge \mathrm{pr}^*e_1 \wedge Id[\bar{\ssf}_{|\h|}(r)]_{1,0} \|_{h_1} =O(s^{-1})$. We first trivialize $\n$ using orthonormal complex coordinates $(x,y)$ and
 taking into account Lemma \ref{lem:bounded norm} we obtain 
$ \|Id[\bar{\ssf}_{|\h|}(r)]_{1,0}\|_{h_1} = \| \sum_{j=1}^4 \bar{\ssf}_{|\h|}'(r) \frac{y_j}{r} \eta_j \|_{h_1} = O(r^{-1})$. On the other hand, $\varpi_s(x,y)=\varpi_s(r)$ and thus 
\begin{align*}
[d\varpi_s]_{1,0} =& \sum_{i=1}^4{\varpi_s'(r) \frac{y_i}{r} \eta_i}, \\
[d\varpi_s]_{1,0} =& - \sum_{i=1}^4\sum_{j=1}^3{ \varpi_s'(r) \frac{y_i}{r} A_j^i (x,y)dx_j} .
\end{align*}
Taking into account that $A_j^i(x,y)=O(r)$ we obtain that 
$ \| [d\varpi_\s]_{1,0} \|_{h_1} =O(s^{-1})$,$ \| [d\varpi_\s]_{0,1}\|_{h_1}=O(r s^{-1})$. A multiplication yields the desired estimates. The remaining estimates are obtained by taking derivatives of 
$$
[Id\bar{\ssf}_{|\h|}]_{1,0}=\frac{\bar{\ssf}_{|\h|}'(r)}{r} (-y_1 \eta_2 + y_2 \eta_1 - y_3 \eta_4 + y_4 \eta_3),
$$
and using Lemma \ref{lem:bounded norm}.
Our estimates yield:
$$
\| t^2 d\xi_s \|_{h_t} = O(r^{-2}) +  O(r^{-1}s^{-1}) + t(O(r^{-1})+ O(s^{-1}) ) + t^2O(1) 
$$
Take $s_0$ such that for each $0<t<1$ and $s>s_0$ it holds that 
$| O(r^{-2}) +  O(r^{-1}s^{-1}) + t(O(r^{-1})+ O(s^{-1}))  |< \frac{m}{4}$
on $\frac{s}{8} \leq r \leq \frac{s}{4}$.
Let $s>s_0$ and take $t_s''<t_s$ such that $|t^2O(1)|< \frac{m}{2} $ and
$
\|\P_2^t - \P_1^t \|_{h_t} < \frac{m}{2}
$
 on $\chi^{-1}(\n_{2s})$; this is possible as we argued before. Define the closed form
$$
\P_{3,s}^t = \P_2^t -  \frac{t^2}{4} d\xi_s,
$$
which coincides with $\P_2^t$ if $r\leq \frac{s}{8}$ and with $\chi^*(\widehat \p_{3,s}^t)$ if $r\geq \frac{s}{4}$. On the neck $\frac{s}{8} \leq r  \leq \frac{s}{4}$ we have that:
$$ \| \P_{3,s}^t - \P_1^t\|_{h_t} \leq \| \P_{3,s}^t - \P_2^t \|_{h_t} + \| \P_2^t - \P_1^t \|_{h_t} < m.$$
The statement is therefore proved.
\end{proof}
The map $F_t \circ \chi $ allows us to glue an annulus around the zero section on $(\nu/\j ,\p_2)$ and an annulus around $Q$ on $(P, \widehat{\P}_2^t)$; this yields a resolution.
\begin{theorem} \label{theo:resol}
There exists a closed $\Gtwo$ resolution $\rho \colon \widetilde{X} \to X$. In addition,
let us denote $D_s(Q)$ the $s$-disk of $P$ centered at $Q$; then
$$
\widetilde X= X- \exp(\nu_{\e}/\j) \cup_{\exp \circ F_t \circ \chi} D_s(Q)
$$
for some $\e>0$, $t>0$ and $s>0$.
\end{theorem}
\begin{proof}
Let $\e_0<R$ and $s_0>0$ be the values provided by Proposition \ref{prop:interpolation-order-2} and \ref{prop:inter-resol}. Fix $s>s_0$ and choose $t<t_s'''$ with $st = \frac{\e}{4}$ for some $\e<\e_0$. The map $F_t \circ \chi$ identifies $s\leq r \leq 2s$ on $P$ with $\frac{\e}{4} \leq r \leq \frac{\e}{2}$ on $\n/\j$.

Consider the $\Gtwo$ forms $\P_{3,s}^t$ on $\chi^{-1}(\n_{2s}/\j)$ and  $\p_{3,\e}$ on  on $\n_{2\e}/\j$; on the annulus $s \leq r \leq 2s$ of $\chi^{-1}(\n_{2s}/\j)$ we have that $\P_{3,s}^t = \chi^*(\widehat \p_{3,s}^t)= \p_2^t$ and on $\frac{\e}{4} \leq r \leq \frac{\e}{2}$ on $\n/\j$ we have that $\p_{3,\e}=\p_2$.

Being $(F_t\circ \chi)^* \p_2 = \chi^*(\p_2^t)$, the $\Gtwo$ structure is well defined on the resolution.
\end{proof}

\begin{remark} \label{rem:size-exc}
The radious of the disc $r \leq 2s$ with respect to the metric $h_t$ is $2st$. 
Fixed $s_0>0$ the map $F_t \circ \chi$ identifies $0<r \leq 2s$ on $P$ with $0<r \leq 2st$ on $\n$; therefore if we choose $t\to 0$ then the size of the exceptional divisor decreases.
\end{remark}

\section{Topology of the resolution} \label{sec:topo}

This section is devoted to understanding the cohomology algebra
of the resolution; we shall make use of real coefficients and
denote by $H^*(M)$ the algebra $H^*(M,\RR)$. We start by describing $H^*(\widetilde X)$
in terms of $H^*(X)$ and $H^*(L)$ and we then compute the induced
product on it.

The fibre bundle $\nu$ is topologically trivial; this follows from the fact that 
every $3$ manifold is parallelizable. For a proof see \cite[Remark 2.14]{JK}. 
However, it might not be trivial as a complex bundle as we shall
deduce from the computation of its total Chern class. 

Let us suppose for a moment that $L$ is connected; then $L$ is the mapping torus of diffeomorphism $\psi \colon \S \to \S$, where $\S$ is an
orientable surface of genus $g$. In section \ref{sec:resolu} we denoted  by  $\rqr \colon \S \x [0,1] \to L$ the quotient projection, and by $\rb \colon L \to S^1$ the bundle projection. We also chose that $\h= \rb^*(\h_0)$ with $\h_0$ the angular form on $S^1$. 

In Proposition \ref{prop:chern-class} we compute the total Chern class of $\n$ by observing first that $\n$ admits a section and thus $\n= \underline{\CC}\oplus \ker{\h}$; where $\underline{\CC}$ denotes the trivial line bundle over $L$. Then we identify $\ker(\h)$ with the tangent space of the fibres taking into account that $\h= \rb^*(\h_0)$. A formula for $c(\n)$ follows from these remarks.

In order to state the result it shall be useful to note that $2$-forms on $\S$ determine closed $2$-forms on $L$. More precisely, let us consider $\varpi \colon [0,1]\to \RR$ a bump function with $\varpi|_{[0,1/4]}=0$ and $\varpi|_{[3/4,1]}=1$. Let $\b \in \O^2(\S)$ and let 
 $\a \in \O^1(\S)$ such that  $\psi^*\b= \b + d\a$; note that this is possible because $\psi^*= \rId$ on $H^2(\S)$. Then $\bar{\b}= \b + d(\varpi(t) \a) \in \O^2(\S \x [0,1])$ induces a $2$-form on $L$ via the push-forward. Of course, one can show that the cohomology class of $\overline{\b}$ does not depend on $\a$. In addition, from the Mayer-Vietoris long exact sequence we deduce that $[\rqr_*(\bar{\b})] \neq 0$ if $[\b] \neq 0$. 

We denote by $\o_\S\in \O^2(L)$ a closed $2$-form induced by a volume form $\mathrm{vol}_\S$ of $\S$ that integrates $1$ on $\S$. This class represents the 
 Poincar\' e dual of a circle $C\subset L$ such that  $\rqr ( \{p_0\} \x [0,1] ) \subset C$ and $C-\rqr ( \{p_0\} \x \{0\})$ is an embedded line on $\rqr(\S \x \{0\})$ if it is not empty.

\begin{proposition} \label{prop:chern-class}
The total Chern class of $\nu$ is $c(\nu)= 1 + (2-2g)[\o_\S]$.
\end{proposition}
\begin{proof}
Let $\x$ be the cross product on $TM|_L$ determined by $\vp$. Consider on $E=\ker(\h)$ the complex structure
$\rJ W = W \x e_1^\sharp$,
where $e_1= \|\h\|^{-1}\h$. This is well-defined because $\x$ defines a cross product on $T_pL$ and if $\h(X)=0$, then $X \x e_1^\sharp \perp e_1^\sharp$.
Recall also that the complex structure on $\nu$ is: $I(v) = e_1^\sharp \x v$.

We prove that there is an isomorphism of complex line bundles:
 $$
 \underline{\CC} \oplus  E \to \nu.
 $$

A nowhere-vanishing section  $s \colon L \to\n$ exists because $\dim L=3>4 = \mathrm{rk}(\n)$; we define the isomorphism $\underline{\CC} \oplus E  \to \nu$,
$$ 
(z_1 + iz_2,W) \longmapsto z_1 s + z_2  e_1^\sharp \x s +  W \x s .
$$
In order to check that the isomorphism is complex linear one uses the equality \cite[Lemma 2.9]{SW}:
 $$
 u \x (v \x w) + v \x (u \x w) = g(u,w)v + g(v,w)u - 2g(u, v)w.
 $$
 where $g$ denotes the restriction to $\n$ of the metric on $M$. In our case taking $u=e_1^\sharp$, $v=s$ and $w=W$ we obtain that $e_1^\sharp \x(W \x s)= ( W \x e_1^\sharp)\x s$.
  
From the isomorphism we get that
$c(\nu)=c(\underline{\CC})c(E)=1+c_1(E)$.  We now compute $c_1(E)$; note that $E$ is the vertical distribution $d\rqr(T\S\x[0,1])\subset TM$.
First consider a compactly-supported $2$-form $\ups\in \O^2(T\S)$ representing the Thom class of the bundle $T\S \to S$ that integrates $1$ over the fibres. Being the diffeomorphism $d\psi \colon T\S \to T\S$ volume-preserving we obtain that $(d\psi)^*\ups$ is also a compactly-supported $2$-form that integrates $1$ over the fibres. Thus, $(d\psi)^*\ups= \ups + d\a$ for some compactly-supported $\a \in \O^1(T\S)$. 
In addition let $s_0 \colon \S \to T\S$ the zero section; then $[s_0^*(\ups)]=(2-2g)[\mathrm{vol}_\S]$.

The push-forward $\rqr_*(\ups + d(\varpi \a)) \in \O^2(E)$ of course induces the Thom class of $E$. Being $s[p,t]=d\rqr_{(p,t)}(s_0(p,t))$ the zero section of $E$ we obtain:
 $$
 c_1(E)=s^*[\rqr_*(\ups + d(\varpi \a))]=[q_*(s_0^*\ups + d(\varpi s_0^* \a) ) ]=(2-2g)[\o_\S].
 $$ 
To obtain the last equality we have taken into account that $s_0^*(d\varpi)=0$, $s_0^*(\psi^* \ups)=s_0^*\ups + d(s_0^* \a)$ and $[s_0^*(\ups)]=(2-2g)[\mathrm{vol}_\S]$.
 
\end{proof}

The projectivized bundle of $\n$ coincides with $Q$ because $
\PP(\n)= P_{\UU(2)}(\n)\x_{\UU(2)} \CP^1 = Q.
$
An obstruction-theoretic argument ensures that it is trivial:

\begin{lemma} \label{lem:Q-trivial}
The bundle $Q \to L$ is trivial.
\end{lemma}
\begin{proof}
First recall that the spaces $\Diff(S^2)$ and $\SO(3)$ have the same homotopy type. Classifying $S^2$ bundles is therefore equivalent to classifying rank $3$ vector bundles.
In our case, denoting by $E=\ker(\h)$ as in the proof of Proposition \ref{prop:chern-class},
if $g_{\a \b} \in SO(2)$ are the transition functions of $E$, taking into account the diffeomorphism $\CP^1 \to S^2$ one can compute that the transition functions of $Q$ are
$$
h_{\a \b}(x)(v_1,v_2,v_3) = (g_{\a \b}(v_1,v_2),v_3)
$$
Therefore, the associated rank $3$ vector bundle $V$ has transition functions $g_{\a \b} \x \rId \in \SO(3)$. This
is trivial if and only if $Q$ is. We now observe that $V$ is trivial if and only if its second Stiefel-Withney class vanishes. For that purpose consider a CW-decomposition,
$$
L=\cup_{k=0}^3{L^k}.
$$ 
Then $V|_{L^1}$ is trivial because $\SO(3)$ is connected. The trivialization extends to $L^2$ if the primary obstruction cocycle is exact; this coincides with the second Stiefel-Whitney class (see \cite[Proposition 3.21]{Hatcher}). If it vanishes, then the last obstruction cocycle lies in $H^3(L,\pi_2(\SO(3)))=0$ and therefore the trivialization extends to $L$.

We now compute the second Stiefel-Whitney class of $V$. Regarding the transition functions $V=E\oplus \RR$ and thus $w_2(V)=w_2(E)$. Being $E$ a complex vector bundle, we obtain $w_2(E)=c_1(E) \mbox{ (mod 2)}=(2-2g)\o_\S  \mbox{ (mod 2)}=0$.
\end{proof}

Using Proposition \ref{prop:chern-class} we re-state a well known fact. For that purpose consider
the tautological bundle associated to $\n$:
$$
\overline{P} = P_{\UU(2)}(\n)\x_{\UU(2)} \widetilde{\CC}^2.
$$
Denote frames in $P_{\UU(2)}(\n)$ by $F$. There is a well-defined $\Z_2$ action on $\overline{P}$, determined 
by $[F,(z_1,z_2,\ell)] \longmapsto [F,(-z_1,-z_2,\ell)]$. The quotient $\overline{P}/\ZZ_2 $ coincides with $P$.
We denote by $\varrho \colon \overline{P} \to P$ the projection.

\begin{proposition}\label{prop:cohomology-exceptional-divisor}
Let $\mathrm{e}(\overline{P})$ be the Euler class the line bundle $\overline{P} \to Q$. Denote by $H^*(L)[\bx]$ the algebra of polynomials with coeffiecients in $H^*(L)$
The map:
$$
F \colon H^*(L)[\bx]/\la \bx^2+ (2-2g) [\o_\S] \bx \ra \to H^{*}(Q),\quad  F(\b)= \mathrm{pr}^*\b, F(\bx)=\mathrm{e}(\overline{P}),
$$
is an isomorphism of algebras.
\end{proposition}

Recall that we denoted the projection by $\mathrm{pr} \colon P \to L$.
Consider $\t \in \O^2(\overline{P})$ the Thom $2$-form of the line bundle $\overline{P} \to Q $ and
note that we can suppose that $\t$ is $\Z_2$-invariant because the involution preserves the orientation on the fibres. From Proposition \ref{prop:cohomology-exceptional-divisor} we obtain:
$$
[\t \wedge \t] = - (2-2g)[(\varrho \circ \mathrm{pr})^*\o_\S \wedge \t].
$$
We also denote by $\t$ the pushforward $\varrho_*\t \in \O(P)$; on $H^*(P)$ it also verifies that:
$$
[\t \wedge \t ]= - (2-2g)[ \mathrm{pr}^*\o_\S \wedge \t].
$$
Of course, we can extend $\t$ to a $2$-form on $\widetilde{X}$ and it corresponds to the Poincar\' e dual of $Q$.

We now compute the cohomology of $\widetilde{X}$; for this we do not assume that $L$ is connected and we denote by $ L_1 ,\dots, L_r $ it connected compontents. Each $L_i$ is the mapping torus of diffeomorphism $\psi_i \colon \S_i \to \S_i$, where $\S_i$ is an
orientable surface of genus $g_i$; we denote by $\o_i$ the  $2$-form $\o_{\S_i}$ as constructed before.
We also denote $Q_i= Q|_{L_i}$, $P_i=P|_{L_i}$ and $\t_i$ the Thom form of $Q_i \subset P_i$. 

\begin{proposition}\label{prop:cohomology-short-sequence}
There is a split exact sequence:
$$
\xymatrix{ 0 \ar[r] & H^*(X) \ar[r]^{\pi^*} & H^*(\widetilde X) \ar[r] & \oplus_{i=1}^r H^*(L_i)\otimes \la \bx_i \ra \ar[r] & 0}
$$
where $\bx_i$ has degree two.

\end{proposition}
\begin{proof}
The existence of such exact sequence is contained in the proof of \cite[Proposition 6.1]{JK}; we outline it. Consider the long exact sequence of pairs $(X, L)$ and $(\widetilde X, Q)$. There is a commutative diagram:
\small{
$$
\xymatrix{   
H^{k}(X,L) \ar[r] \ar[d]^{\pi^*}  &
 H^{k}(X) \ar[r]^{e_L^*} \ar[d]^{\pi^*}  &
 \oplus_i H^{k}(L_i,\RR) \ar[d]^{\pi^*} \ar[r]^{D_1} &
  H^{k+1}(X, L ) \ar[d]^{\pi^*} \\
  H^{k}(\widetilde X,Q) \ar[r] & 
 H^{k}(\widetilde X) \ar[r]^{e_Q^*} &
   \oplus_i H^{k}(Q_i) \ar[r]^{D_2} & 
   H^{k+1}(\widetilde X, Q ) 
}
$$ }
\normalsize{
Here we denoted the inclusions $e_L \colon L \to X$ and $e_Q \colon Q \to \widetilde{X}$.
 The first and fourth columns are isomorphisms; these correspond to the identity map. The third column is injective with cokernel $\oplus_i H^*(Q_i)/H^*(L_i)$; this is isomorphic to $\oplus_i H^{k-2}(L_i)\otimes \la \bx_i\ra $, because $Q_i=L_i \times S^2$. Thus we get a commutative diagram with exact columns:}

\small{
$$
\xymatrix{  &0 \ar[d] &0 \ar[d] &  \\
 H^{k}(X,L) \ar[r] \ar[d]^{\pi^*}  & H^{k}(X) \ar[r]^{e_L^*} \ar[d]^{\pi^*}  & \oplus_i H^{k}(L_i) \ar[d]^{\pi^*} \ar[r]^{D_1} & H^{k+1}(X, L ) \ar[d]^{\pi^*} \\
  H^{k}(\widetilde X,Q) \ar[r] & H^{k}(\widetilde X) \ar[r]^{e_Q^*} \ar[d] & \oplus_i H^{k}(Q_i) \ar[r]^{D_2} \ar[d] & H^{k+1}(\widetilde X, Q ) \\
   & \coker(\pi^*) \ar[d] \ar[r]^{\bar{e}_Q} &  \oplus_i H^{k-2}(L_i)\otimes \la \bx_i\ra \ar[d] &  \\
   &0  &0  &    
}
$$ }
\normalsize{
Of course, $\bar{e}_Q$ is the action induced by $e_Q^*$ on the quotient. In addition, the fact that first and fourth columns are the identity implies that $\mathrm{Im}(e_L^*)=\mathrm{Im}(e_Q^*)$.}

 Snake Lemma ensures that there is a exact sequence:
$$
0 \to \ker(e_L^*) \to \ker(e_Q^*) \to \ker(\bar{e}_Q) \to \coker(e_L^*) \to \coker(e_Q^*) \to \coker(\bar{e}_Q) \to 0.
$$
The maps are induced by $\pi^*$, except for the connecting map $\ker(\bar{e}_Q) \to \coker(e_L)$. The map $\pi^* \colon \ker(e_L^*) \to \ker(e_Q^*)$ is an isomorphism because the first column is an isomorphism and the diagram is commutative. In addition, taking into account that the fourth column is an isomorphism and that the diagram is commutative one can also check that $\pi^*$ is an isomorphism between $\mathrm{Im}(D_1)$  and $\mathrm{Im}(D_2)$.
Moreover:
$$
\mathrm{Im}(D_1)= \oplus_i H^*(Q_i)/ \ker (D_1) = \oplus_i H^{*}(L_i)/ \mathrm{Im}(e_L^*) = \coker(e_L^*),
$$
and the isomorphism is induced by the map that $\pi^*$ induces on the quotient. Simmilarly, $\coker(e_Q^*)$ is isomorphic to $\mathrm{Im}(D_2)$ via $\pi^*$.
This means that $\ker(\bar{e}_Q)=0=\coker(\bar{e}_Q)$ so,
$$
\coker (\pi^*)= \oplus_i  H^{*-2}(L_i)\otimes \la \bx_i \ra.
$$

Consider $\t_i$ the Poincar\' e dual of $Q_i \subset \widetilde{X}$ as constructed before. Then,
$$
\b \otimes \bx_i \longmapsto \mathrm{pr}^*(\b) \t_i
$$
is a splitting of the previous exact sequence.
\end{proof}

This result implies that there is an isomorphism of vector spaces between $H^*(\widetilde{X})$ and  $H^*(X) \oplus \oplus_{i=1}^r  H^*(L_i)\otimes \la \bx_i \ra $. The algebra structure of $H^*(\widetilde{X})$ induces an algebra structure on $H^*(X) \oplus \oplus_{i=1}^r H^*(L_i)\otimes \la \bx_i \ra $ that we compute in Proposition \ref{prop:cohom-alg}. This is necessary in order to decide whether the resolution $\widetilde{X}$ is formal or not, because formality condition involves products of cohomology classes.

\begin{proposition} \label{prop:cohom-alg}
There is an isomorphism
$$
 H^*(\widetilde{X})= H^*(X) \bigoplus \oplus_{i=1}^r  H^*(L_i)\otimes \la \bx_i \ra.
$$
Let $\a, \b \in H^*(X)$, $\g_i \in H^*(L_i)$, $\g_j' \in H^*(L_j)$ and let $e_i \colon L_i \to X$ be the inclusion. The wedge product on $ H^*(\widetilde{X})$ determines the following product on the left hand side:  
\begin{enumerate}
\item $\a \b = \a \wedge \beta$,
\item $\a (\g_i\otimes \bx_i) = (e_i^*(\a) \wedge \g_i) \otimes \bx_i$,
\item $(\g_i\otimes \bx_i) (\g_j' \otimes \bx_i)=0$ if $i\neq j$,
\item $(\g_i \otimes \bx_i)( \g_i' \otimes \bx_i) = -2(\g_i \wedge \g_i')PD[L_i] -(2-2g_i)( \o_i \otimes \bx_i)  $.
\end{enumerate}
\end{proposition} 
\begin{proof}
Let $s \colon  \oplus_{i=1}^r H^*(L_i)\otimes \la \bx_i \ra \to H^*(\widetilde{X})$ be the splitting map constructed in the proof of Proposition \ref{prop:cohomology-short-sequence}. Then, the isomorphism is determined by:
$$
\rT =(\rho^*,s) \colon H^*(X) \oplus \oplus_{i=1} H^*(L_i)\otimes \la \bx_i \ra \to H^*(\widetilde{X}).
$$
In order to obtain a formula for the product between forms $\eta$, $\eta'$ we have to compute
$(\rT)^{-1} \left( \rT \eta \wedge \rT \eta' \right)$.
All the statements are evident except for the last one. We only check $\bx_i^2= -2PD[L_i] - (2-2g_i)(\o_i\otimes \bx_i)$, the announced formula is deduced from this and the fact that $H^*(\widetilde X)$ is an algebra. First of all, $\rT \bx_ i \wedge \rT \bx_i = [\t_i \wedge \t_i]$; we now compute $\rT^{-1}[\t_i \wedge \t_i]$. On the one hand taking into account the equality
$$
[ \t_i \wedge \t_i ] = - (2-2g_i)[\mathrm{pr}^*(\o_i) \wedge  \t_i],
$$ 
 we obtain that the restriction of $\rT^{-1}[\t_i \wedge \t_i]$ to $H^*(L_i)\otimes \la \bx_i \ra$ is
$
 - (2-2g_i)(\o_i \otimes \bx_i) .
$
On the other hand, note first that if $x\in L_i$ then $\t_i|_{P_x}$ is the Thom form of $Q_x \subset P_x$ because $\t_i$ is the Thom form of $Q_i \subset P_i$. Thus:
$$
\int_{P_x}{\tau_i \wedge \tau_i}= [Q_x][Q_x]=-2.
$$
The restriction of $\rT^{-1} [\t_i \wedge \t_i]$ to $H^*(X)$  has compact support around $L_i$ and
$$
\int_{\n_x}{ \rho^*(\tau_i \wedge \tau_i)}= \int_{\n_x -0}{ \rho^*(\tau_i \wedge \tau_i)}=
\int_{P_x-Q_x }{ \tau_i \wedge \tau_i}= \int_{P_x }{ \tau_i \wedge \tau_i}=-2.
$$
The restriction in thus equal to $-2PD[L_i]$.
\end{proof}

\section{Non-formal compact $\Gtwo$ manifold with $b_1=1$} \label{sec:const}

Nilpotent Lie algebras that have a closed left-invariant $\Gtwo$ structure are classified in \cite{CF}; from these one obtain nilmanifolds with an invariant closed $\Gtwo$ structure. Of course, excluding the $7$-dimensional torus, these are non-formal and have $b_2 \geq 2$. From a $\Z_2$ action on a nilmanifold, in \cite{FFKM} authors construct a formal orbifold whose isotropy locus are $16$ disjoint $3$-tori; then they prove that its resolution is also formal. In this section we follow the same process to construct first a non-formal $\Gtwo$ orbifold with $b_1=1$ from a nilmanifold; its isotropy locus consists of $16$ disjoint non-formal nilmanifolds. Later we prove that its resolution is also non-formal and does not admit any torsion-free $\Gtwo$ structure.

\subsection{Orbifold with $b_1=1$}

Let us consider the Lie algebra $\frg$ with structure equations
$$
(0,0,0,12,23,-13,-2 (16)+ 2 (25) + 2(26) - 2(34) ),
$$
and let $(e_1,e_2,e_3,e_4,e_5,e_6,e_7)$ be the generators of  $\frg$ that verify the structure equations, that is, $[e_1,e_2]=-e_4$, $[e_2,e_3]=-e_5$ and so on.
Recall that the simply connected Lie group $G$ associated to $\frg$ is the vector space $\frg$ endowed with the product $*$ determined by the Baker-Campbell-Hausdorff formula.

\begin{remark}
The Lie algebra $\frg$ belongs to the $1$-parameter family of algebras $147E1$ listed in Gong's classification  \cite{G}; we choose the parameter $\l=2$. The associated Lie group admits an invariant closed $\Gtwo$ structure as proved in \cite{CF}.
\end{remark}

Define $u_1=e_1$, $u_2=e_2$, $u_3=e_3$, $u_4= \frac{1}{2}e_4$, $u_5= \frac{1}{2}e_5$, $u_6= \frac{1}{2}e_6$ and
$u_7= \frac{1}{6}e_7$. 
\begin{proposition}\label{prop:gh}
If $x=\sum_{k=1}^7{\l_k u_k}$ and $y= \sum_{k=1}^7{\m_k u_k}$ then
\begin{align*} 
x*y=  &( \l_1 + \m_1)u_1 + (\l_2 + \m_2)u_2 + (\l_3 + \m_3)u_3 + (\l_4 + \m_4 - ( \l_1\m_2 - \l_2\m_1))u_4 \\
    &+ (\l_5 + \m_5 - ( \l_2\m_3 - \l_3\m_2))u_5 + (\l_6 + \m_6 + (\l_1\m_3 - \l_3\m_1))u_6 \\
 &+ (\l_7 + \m_7 + (\l_1-\m_1 - \l_2 + \m_2)(\l_1 \m_3 - \l_3 \m_1) - (\l_3- \m_3)(\l_1 \m_2 - \m_2 \l_1 ))u_7 \\
 &+ (- (\l_2 - \m_2)(\l_2\m_3-\l_3 \m_2)  + 3 (\l_1 \m_6 +  \l_6 \m_1))u_7 \\
   &+  (- 3 (\l_2 \m_5 -  \l_5 \m_2) - 3( \l_2 \m_6 - \l_6 \m_2) + 3(\l_3 \m_4 +  \l_4\m_3)  )u_7.
\end{align*}
\end{proposition}
\begin{proof}
Being $\frg$ is $3$-step, the Baker-Campbell-Hausdorff formula yields:
$$
x*y = x + y  + \frac{1}{2} [x,y] + \frac{1}{12} \left( [x,[x,y]] - [y,[x,y]] \right).
$$
Taking into account that $u_7\in Z(\frg)$ and that $ [u_i,[u_j,u_k]]= 0$ if $i\geq 4$ or $j \geq 4$ or $k \leq 4$, it follows:
 \begin{align*}
x*y =&
\sum_{k=1}^7{(\l_k+ \m_i) u_i} +   \frac{1}{2} \sum_{1 \leq i<j \leq 7}(\l_i\m_j - \l_j\m_i)[u_i,u_j] \\
&+ \frac{1}{12} \sum_{1 \leq k\leq 3} (\l_k-\m_k) \sum_{1 \leq i<j \leq 3} {(\l_i\m_j - \l_j\m_i)[u_k,[u_i,u_j]]};
\end{align*}
The non-zero combinations $[u_i,u_j]$ and $[u_k,[u_i,u_j]]$ are:
\begin{align*}
[u_1,u_2]=& -2u_4, & [u_2,u_5]=& -6u_7, &  [u_3,[u_1,u_2]]=& -12 u_7 \\
[u_1,u_3]=& 2 u_6, & [u_2,u_6]=& -6u_7, &  [u_1,[u_1,u_3]]=& 12 u_7\\
[u_1,u_6]=& 6u_7, &  [u_3,u_4]=&6u_7, &  [u_2,[u_1,u_3]]=&  -12 u_7 , \\
[u_2,u_3]=& -2u_5, & & & [u_2,[u_2,u_3]]=& 12u_7. 
\end{align*}
The announced formula easily follows from this.
\end{proof}

Proposition \ref{prop:gh} ensures that
$$
\G= \Big\{ \sum_{i=1}^7{n_i u_i}, \mbox{ s.t. } n_i \in \Z \Big\},
$$
is a discrete subgroup of $G$, which is of course co-compact. 
Indeed, a straightforward computation gives a fundamental domain for the left action of $\G$ on $G$:

\begin{proposition} \label{prop:lattice}
 A fundamental domain for left the action of $\G$ on $G$ is 
$$
\mathcal{D}=\Big\{  \sum_{i=1}^7{t_i u_i}, \mbox{ s.t. }  0 \leq t_i \leq 1 \Big \}.
$$
\end{proposition}

According to \cite[Lemma 5]{CF}, the group $G$ admits an invariant closed $\Gtwo$ structure 
determined by:
$$
\vp= v^{127} + v^{347} + v^{567} + v^{135} - v^{236} - v^{146} - v^{245}.
$$ 
where:

\begin{minipage}[t]{0.48\textwidth}
\begin{itemize}
\item $v^1=\sqrt{3}(2e^1 + e^5 - e^2 + e^6)$;
\item $v^2=3e^2 - e^5 + e^6$;
\item $v^3=e^3 + 2e^4$;
\item $v^4=\sqrt{3}(e^3 + e^7)$;
\end{itemize}
\end{minipage}
\begin{minipage}[t]{0.48\textwidth}
\begin{itemize}
\item $v^5=\sqrt{2}(e^6 - e^5)$; 
\item $v^6=\sqrt{6}(e^5 + e^6)$,
\item $v_7=2\sqrt{2}(e^4 - e^3)$.
\end{itemize}
\end{minipage}

Consider $M=G/\G$;  points
of $M$ will be denoted by $[x]$, for some $x \in G$. The nilmanifold $M$ inherits a closed $\Gtwo$ structure that we also denote by $\vp$.
We now define an involution $\j$ on $M$ such that $\j^*\vp= \vp$. For that purpose it is
sufficient to define an order $2$ isomorphism $\j \colon G \to G$ of $G$ with $\j^*\vp=\vp$, and $\jota \G = \G$. The desired map is:
$$
\j(e_k)=e_k, \quad k \in{3,4,7}, \qquad \j(e_k)=-e_k, \quad k\in \{1,2,5,6\}.
$$
Looking at the structure constants of $G$ it becomes clear that $\j$ is
an automorphism of $\frg$. Baker-Campbell-Hausdorff formula ensures
that $\j$ is an homomorphism. In addition, it is clear that $\j(\G)\subset \G$.
Finally, one can easily deduce that $\j^*(\vp)= \vp$.

We define the orbifold $X=M/\j$, which has a closed $\Gtwo$ structure determined by $\vp$. We
now study its singular locus:

\begin{proposition}
The isotropy locus has $16$ connected components; these are all diffeomorphic and their universal
covering is the Heisenberg group.
Let us define $H_0=\{ \l_3 u_3 + \l_4 u_4 + \l_7 u_7, \mbox { s.t. } \l_j \in \RR  \}$ and 
$\mathcal{E}=\{ \e_1 u_1 + \e_2 u_2 + \e_5 u_5 + \e_6 u_6, \mbox{ s.t. } \e_j \in \{0, \frac{1}{2}\} \}$. The $16$ connected components of the isotropy locus are:
$$
H_\e=[L_{\e}H_0], \quad \e \in \mathcal{E},
$$
where $L_\e$ denotes the left translation on $G$ by the element $\e \in \mathcal{E}$.
\end{proposition}
\begin{proof}
It is clear that $H_0$ is a connected component of $\Fix(\j)$ that contains $0$, which is the unit of $G$. Being $\j$ an homomorphism, we conclude that $H_0$ is a subgroup of $G$. It is thus sufficient to prove that the Lie algebra $\frh$ of $H_0$ is the Heisenberg algebra. This is of course true because
 $\frh = \la e_3, e_4, e_7 \ra$ with $[e_3,e_4]=e_7$ and $[e_j,e_7]=0$ for $j\in \{3,4\}$.
 
Let $\mathcal{K}= \{1,2,5,6\}$ and consider $x= \sum_{k\in \mathcal{K}}{\l_k u_k} \in \mathcal{D}$, that is, $\l_k \in [0,1]$. We now check that if $\g*x = \j(x)$ for some $\gamma \in \Gamma $ then $[x] \in H_\e$ for some $\e \in \mathcal{E}$. Let us denote $\gamma = \sum_{k=1}^{7}{n_k u_k}$; taking into account Proposition \ref{prop:gh} one obtains:
 \begin{align*}
 \g*x =& (n_1 + \l_1)u_1 + (n_2 + \l_2)u_2  + n_3u_3 +  (n_4 - n_1 \l_2 + n_2 \l_1)u_4 \\
  &+  (n_5 + \l_5 + n_3 \l_2)u_5 + (n_6 + \l_6 - n_3 \l_1 )u_6 + \l' u_7,
 \end{align*}
 for some $\l' \in \RR$. The equation $\j(x)= \g*x$ yields immediatly to $2\l_j =- n_j$ for $j=\{1,2\}$ and $n_3=0$. 
 Taking this into account, $n_4 - n_1 \l_2 + n_2 \l_1= n_4$, $n_5 + \l_5 + n_3 \l_2= n_5 + \l_5$, $n_6 + a_6 - n_3 \l_1= n_6 + \l_6$ and thus $n_4=0$, $2\l_5=-n_5$ and $2\l_6=-n_6$. 
 Thus, $x= -\frac{1}{2} \sum_{k \in \mathcal{K}} n_k u_k$; and thus $x \in H_\e$ for some $\e \in \mathcal{E}$.
 
 We now let $[y]$ be an isotropy point; one can write:
  $y= x_1*x_2$; with $x_1= \sum_{k\in \mathcal{K}}{\l_k u_k}$ and $x_2=\sum_{k\notin \mathcal{K}}\m_k u_k \in H_0$. 
 The choice becomes clear from the equality:
 \begin{align*}
 x_1*x_2=&
 \l_1 u_1 + \l_2 u_2 + \m_3u_3 +  \m_4 u_4 +
  ( \l_5 - \l_2 \m_3 )u_5 + ( \l_6 + \a_1 \m_3 )u_6 \\
 &+( \m_7+ (\l_1 - \l_2)(\l_1 \m_3) + \l_2 \m_3)u_7,
  \end{align*}
  that is of course deduced from Proposition \ref{prop:gh}.
 
 Using this decomposition we obtain the equality
 $\g*x_1 x_2= \j(y)= \j(x_1) x_2$ that implies $\j(x_1)=\g x_1$. Take $x_1' \in \mathcal{E}$ with $x_1= \g' x_1'$,
 then
 $
 [y]=[\g' x_1' x_2]=[x_1' x_2] \in [L_{x_1'}H_0].
 $
\end{proof}

\subsection{Non-formality of the resolution}

We start by computing the real cohomology algebra of the orbifold. Nomizu's theorem  \cite{Nomizu} ensures that $(\L^* \frg^*,d)$ is the minimal model of $M$. Taking into
account that $H^*(X)=H^*(M)^{\ZZ_2}$ we obtain that $((\L^* \frg^*)^{\ZZ_2}, d)$ is a model for $X$.
The cohomology of $X$ is:
\begin{align*}
H^1(X)=& \la [e^3] \ra,  \\
H^2(X)=& \la [e^{25}], [e^{15}- e^{26}], [e^{15}- e^{34}] \ra, \\
H^3(X)=& \la [e^{235}], [e^{135}], [e^{356}], [e^{124}],  [e^{146}],  [e^{245}], [e^{127} + 2e^{145}], \\
&\,\, [e^{125} + e^{167} - e^{257} - 2 e^{456} - e^{347}]\ra .
\end{align*}
We now prove that $X$ is not formal.
\begin{proposition} \label{prop:X-nf}
The triple Massey product $\la [e^3],[e^{15}- e^{26}],[e^3] \ra$ of $((\L^* \frg^*)^{\ZZ_2},d)$ is not trivial. Therefore, $X$ is not formal.
\end{proposition}
\begin{proof}
First of all, one can check that that space of exact $3$-forms of $((\L \frg)^{\Z_2},d)$ is:
$$
B^3((\L^* \frg^*)^{\Z_2},d)= \la e^{123},e^{135}-e^{236}, -e^{136}+ e^{235} + e^{236}, e^{127}- 2 e^{146} + 2 e^{245} + 2 e^{246}  \ra.
$$
and the space of closed $2$-forms is:
$$
Z^2( (\L^* \frg^*)^{\ZZ_2},d) = \la e^{12},-e^{16}+e^{25}+e^{26}-e^{34}, e^{25}, e^{15}-e^{26}, e^{15}- e^{34} \ra  .
$$

Let us take $\xi_1= [e^3]= \xi_3$, $\xi_2= [e^{15} - e^{26}]$; the representatives of these cohomology classes are $\a_1=\a_3=e^3$ and $\a_2= e^{15} - e^{26} + dx$ for some $x \in (\frg^*)^{\ZZ_2} $; our previous computations ensure that 
the Massey product $\la \xi_1,\xi_2,\xi_3\ra$ is well defined. More precisely, $\bar{\a}_1 \wedge \a_2 = d(-e^{56} + e^3x + \b_1)$ and $\bar{\a}_2 \wedge \a_3 = d( e^{56}-e^3x+ \b_2)$, where $\b_1$ and $\b_2$ are closed forms. Defining systems for $\la \xi_1, \xi_2, \xi_3 \ra$ are $(e^3,e^{15} - e^{26} + dx, e^3, -e^{56} + e^3x + \b_1, e^{56}-e^3x+ \b_2 )$ and the triple Massey product is
$$
\la \xi_1, \xi_2, \xi_3 \ra=\{ [2e^{356} + e^{3}\b] \mbox{ s.t. } d\b =0 \}.
$$
The zero cohomology class is not an element of this set due to our previous computations.
 Corollary \ref{cor:massey-form} ensures that $X$ is not formal.
\end{proof}

Let $\rho \colon \tilde{X} \to X$ be the closed $\Gtwo$ resolution constructed in Theorem \ref{theo:resol}.
Lifting this triple Massey product to $\widetilde X$ we prove that $\widetilde X$ is not formal.

\begin{proposition} \label{prop:resol-nf}
The resolution $\widetilde X$ is not formal.
\end{proposition}
\begin{proof}
Let $(\L V,d)$ be the minimal model of $\widetilde X$ with $V=\oplus_{i=1}^7 V^i$, and let $\k \colon \L V \to \O(\widetilde X)$ be a quasi-isomorphism.
From Proposition \ref{prop:cohom-alg} we deduce that $H^1(\widetilde X)= \la \rho^*(e^3) \ra$ and that:
$$
H^2 (\widetilde X)= \la \rho^*(e^{25}), \rho^*(e^{15}-e^{26}), \rho^*(e^{15}-e^{34}), \t_1,\dots, \t_{16} \ra.
$$
In addition, $\rho^*(e^3 \wedge (e^{15}-e^{26})) = d\rho^*(e^{56})$ and $\rho^*[e^{235}]$ and $\rho^*[e^{135}]$ are linearly independent
on $H^3(\widetilde X, \RR)$. Then, according to Proposition \ref{prop:cohom-alg} one can choose:
\begin{align*}
V^1=& \la a\ra ,\\
V^2=& \la b_1,b_2,b_3,y_1, \dots, y_{16}, n \ra.
\end{align*}
with $da=0$, $db_j=dy_j=0$ and $dn=a b_2$ and the map $\k$ is:
\begin{align*}
\k(a)=&\rho^*(e^3), & \k(b_2)=&\rho^*(e^{15}-e^{26}),& \k(n)=& \rho^*(e^{56}), \\
\k(b_1)=& \rho^*(e^{25}), & \k(b_3)=& \rho^*(e^{15}-e^{34}), & \k(y_j)=& \t_j.
\end{align*}
We now define a Massey product. Let us take $\xi_1= [a]= \xi_3$, $\xi_2= [b_2]$; the representatives  of these cohomology classes are $\a_1=\a_3=a$ and $\a_2= b_2$. Then $\bar{\a}_1\wedge \a_2 = d(-n+ \b_1 + \o_1)$ and $\bar{\a}_2 \wedge \a_3=d(n+ \b_2 + \o_2)$ with $\b_1,\b_2 \in \la b_1,b_2,b_3 \ra$ and $\o_1,\o_2 \in \la y_{1}, \dots, y_{16} \ra$. 
Therefore, defining systems of $\la \xi_1,\xi_2,\xi_3 \ra$ are $(a,b_2,a,-n + \b_1 + \o_1, n+ \b_2 + \o_2)$ and the Massey product is the set
$$
\{ [2an + a\b + a\o] \mbox{ s.t. } \b \in \la b_1,b_2,b_3 \ra, \quad \o \in \la y_{1}, \dots, y_{16} \ra \} .
$$
We now observe that $[2an + a\b + a\o]=0$ in $H^*(\L V,d)$ if and only if $\o=0$ and $[\k(2an + a\b)]=0$. 
This is because $[\k(a \o)]=[\rho^*(e^3) \wedge \k(\o)]=0$ if and only if $\o=0$, and if $[\o]\neq 0$, the elements
$[\k(a \o)]$ and $[\k (2an + a\b)]$ are linearly independent.

In addition, $\k(2an + a\b)=\rho^*(2e^{356} + e^3 \wedge \b')$, with $\b' \in \la  e^{25}, e^{15}-e^{26},  e^{15}- e^{34} \ra$.
Taking into account Proposition \ref{prop:cohom-alg} $[\k(2an + a\b]=0$ if and only if $[2e^{356} + e^3 \wedge \b']=0$ on $X$. But $[2e^{356} + e^3 \wedge \b']\neq 0$ as shown in Proposition \ref{prop:X-nf}.
\end{proof}

There is another non-trivial triple Massey product that comes from the isotropy locus. In order to describe it we
have to construct the subspace $V^3$ of our minimal model; it is a direct sum $V^3=C\oplus N$; such that $dC=0$ and there are not closed elements on $N$. To construct $C$ one takes a
 basis of the space  
$H^3(\widetilde{X}) / H^1(\widetilde X)H^2(\widetilde X)$; for instance:
\begin{align*}
\la &\rho^*[e^{346}], \rho^*[e^{124}],  \rho^*[e^{146}],  \rho^*[e^{245}], \rho^*[e^{127} + 2e^{145}], \\
&\rho^*[e^{125} + e^{167} - e^{257} - 2 e^{456} - e^{347}]\ra 
\oplus \la \{ [e^4 \otimes \bx_i ] \}_{i=1}^{16}\ra,
\end{align*}
Let $C= \la c_1, \dots, c_6, z_1, \dots z_{16} \ra$ with $dC=0$
and define $\k(c_1)=\rho^*(e^{346}),\k(c_2)=\rho^*(e^{124})$, $\dots$, $\k(c_6)=\rho^*(e^{125} + e^{167} - e^{257} - 2 e^{456} - e^{347})$ and $\k(z_i)=e^4\otimes \bx_i$.

With this notation, the triple Massey product coming from the singular locus
$$
\la [a],[z_j],-[a] \ra
$$
is not trivial.

\begin{proposition}
The fundamental group of $\widetilde X$ is $\pi_1(\widetilde X)=\Z \x \Z_2 \x \Z_6$.
\end{proposition}
\begin{proof}
Let us denote $\pi \colon M \to X$ the quotient projection.
In order to compute $\pi_1(X)$ we first observe that $\pi_1(M)$ is isomorphic to $\Gamma$ due to the exact sequence
$
0 \to \pi_1(G) \to \pi_1(M) \to \Gamma \to 0.
$
Of course, each generator $u_i \in \Gamma$ is identified with the homotopy class $f_i$
determined by the image of the path
from $0$ to $u_i$ under the quotient map $q \colon G \to M$.
Denote by $[\cdot, \cdot]$ the commutator of two elements on $\pi_1(M)$; then
the product structure on $\G$ determines that the non-zero commutators are:
\begin{align*}
[f_1,f_2]=& f_4^{-2}, &  [f_1,f_2]=& f_5^{-2},  & [f_2,f_5]=&  f_7^{-6}, &[f_3,f_4]=& f_7^{6}. \\
[f_1,f_3]=&  f_6^2, &   [f_1,f_6]=& f_7^{6}, & [f_2,f_6]=& f_7^{-6}, & & 
\end{align*}
Taking into account \cite[Corollary 6.3]{Bre} the map
$\pi_* \colon \pi_1(M) \to \pi_1(X)$ is surjective; we now analyze $\pi_*(f_j)$.
First of all, under the projection $\pi$ the image of the loop $f_1$ is the same as the path from
$0$ to $\frac{1}{2} x_1$ followed by the same path in the reverse direction; this is of course contractible and thus $\pi_*(f_1)=0$; in the same manner $\pi_*(f_2)=\pi_*(f_5)=\pi_*(f_6)=0$.
Taking into account commutator relations this implies that $\pi_*(f_4^2)=0$, $\pi_*(f_7^6)=0$
and that $\pi_*(f_3)$, $\pi_*(f_4)$, $\pi_*(f_7)$ commute.
Thus, $\pi_1(X)=\ZZ \x \Z_2 \x \Z_6$.

We now prove that the resolution process does not alter the fundamental group. For each $\e \in \mathcal{E}$ consider a small tubular neighbourhood $B^\e$ of $H_\e$ and suppose additionally that $B^\e$ are pairwise disjoint. Take $D^\e \subset B^\e$ a smaller tubular neighbourhood of $H_\e$
Define $U$ a connected open set containing $\cup_\e B^\e$ that is homotopy equivalent to $\bigvee_\e H_\e$ and $V=X-\cup_\e D^\e$.

Seifert-Van Kampen theorem states that $\pi_1(X)$ is the amalgameted product of $\pi_1(V)$ and $\pi_1(U)$ via $\pi_1(U \cap V)$.  
Define $\widetilde U=\rho^{-1}(U)$, $\widetilde V=\rho^{-1}(V)$; note that $\widetilde V$ and $V$ are diffeomorphic via $\rho$; in addition, $\rho_* \colon \pi_1(\widetilde U) \to \pi_1(U)$ is an isomorphism because $\widetilde U$ is homotopy equivallent to $\bigvee_\e H_\e\x S^2$. This observation and a further application of Seifert-Van Kampen theorem ensures that $\pi_1(\widetilde{X})=\pi_1(X)$.
\end{proof}

\begin{proposition} \label{prop: no-torsion-free}
The manifold $\widetilde X$ does not admit  torsion-free $\Gtwo$ structures.
\end{proposition}
\begin{proof}
Suppose that $\widetilde X$ admits a torsion-free $\Gtwo$ structure. 
Since $g$ is Ricci flat and $b_1=1$, 
\cite{Bo} ensures that there is a finite 
covering $N\x S^1 \to \widetilde X$; with $N$ a compact simply connected $6$-dimensional manifold. Note that the covering is regular because $\pi_1(\widetilde X)$ is abelian; thus $(N \x S^1) / H =\widetilde X$, where $H$ denotes the Deck group of the covering.

The manifold $N$ is formal because it is simply-connected and $6$-dimensional (see \cite[Theorem 3.2]{FM}  ); therefore $N \x S^1$ is formal (see \cite[Lemma 2.11]{FM}). Lemma \ref{lem:formal-quotient} allows us to conclude that $(N \x S^1) / H =\widetilde X$ is formal; yielding a contradiction.
\end{proof}

\begin{remark}
We can also prove Proposition \ref{prop: no-torsion-free} by making use of the topological obstruction of torsion-free $\Gtwo$ structures obtained in \cite{CKT}. Suppose that $\widetilde X$ has a torsion-free $\Gtwo$ structure, then \cite[Theorem 4.10]{CKT} guarantees the existence of CDGAs $(A,d)$ and $(B,d)$ with the differential $d\colon B^k \to B^{k+1}$ being zero except for $k=3$, and
 quasi-isomorphisms:
$$
\xymatrix{
(\O(\widetilde X),d) & \ar[l] (A,d) \ar[r] & (B,d).
}
$$
This implies \cite[Corollary 4.13]{CKT} that
non-zero triple Massey products $\la \xi_1, \xi_2,\xi_3 \ra$ on $(\O(\widetilde X),d)$ verify that $|\xi_1|+ |\xi_2|=4$ and $|\xi_2|+ |\xi_3|=4$. Let $(A',d)$ be the minimal model of $(A,d)$, then one can obtain quasi-isomorphisms: $$
\xymatrix{
(\L V,d) & \ar[l] (A',d) \ar[r] & (B,d).
}
$$
The same conclusion holds for non-zero Massey products on $(\L V, d)$. This contradicts the fact that there is a non-zero Massey product $\la \xi_1, \xi_2,\xi_3 \ra$ on $(\L V,d)$ with $|\xi_1|=|\xi_3|=1$ and $|\xi_2|=2$ as it is obtained in the proof of Proposition \ref{prop:resol-nf}. Therefore $\widetilde X$ does not have a torsion-free $\Gtwo$ structure.
\end{remark}

\begin{remark}
There exists a finite covering $Y \to \widetilde X$ such that $\pi_1(Y)=\ZZ$ because
$\pi_1(\widetilde X)=\ZZ \x \ZZ_2 \x \ZZ_6$.
The manifold $Y$ is also non-formal as a consequence of Lemma \ref{lem:formal-quotient} and of course, it has first Betti number $b_1=1$ and admits a closed $\Gtwo$ structure.
Argueing as in the proof of Proposition \ref{prop: no-torsion-free} one can conclude that $Y$ does not admit any torsion-free $\Gtwo$ structure.
\end{remark}

\end{document}